\documentclass[11pt]{amsart}
\usepackage{graphicx}
\usepackage{pgf,tikz,pgfplots}
\usepackage{mathrsfs}
\usepackage{mathpple}
\usepackage{graphicx}
\usepackage[all,cmtip]{xy}
\usepackage{amsmath}
\usepackage{tikz}
\usepackage{mathdots}
\usepackage{yhmath}
\usepackage{cancel}
\usepackage{color}
\usepackage{siunitx}
\usepackage{array}
\usepackage{multirow}
\usepackage{amssymb}
\usepackage{gensymb}
\usepackage{tabularx}
\usepackage{booktabs}
\usepackage{makecell}
\usepackage{young}
\usepackage{youngtab}
\usepackage{titletoc}
\usepackage{extarrows}
\usepackage{varioref} 
\usepackage{bm}
\usepackage{mathtools}
\usepackage{ytableau}
\usepackage{array}
\usepackage{multirow}
\usetikzlibrary{fadings}
\usetikzlibrary{patterns}
\usetikzlibrary{shadows.blur}
\usetikzlibrary{shapes}
\usepackage{multicol}
\usepackage[
ordering=Kac,
edge/.style=black,
indefinite-edge={draw=green,fill=white,densely dashed},
indefinite-edge-ratio=5,
mark=o,
root-radius=.06cm]
{dynkin-diagrams}

\usepackage{longtable}
\usepackage[normalem]{ulem}
\usepackage{hyperref,xcolor}
\definecolor{wine-stain}{rgb}{0.5,0,0}
\hypersetup{
	colorlinks,
	linkcolor=wine-stain,
	linktoc=all
}

\usepackage{latexsym,bm,amsmath,amssymb}

\usetikzlibrary{arrows}
\usepackage{todonotes}
\definecolor{pistachio}{rgb}{0.58, 0.77, 0.45}
\definecolor{eggshell}{rgb}{0.94, 0.92, 0.84}

\newtheorem{theorem}{Theorem~}[section]
\newtheorem{lemma}[theorem]{Lemma~}

\newtheorem{proposition}[theorem]{Proposition~}

\newtheorem{corollary}[theorem]{Corollary~}
\newtheorem{proposition-definition}[theorem]{Proposition-Definition~}

\theoremstyle{remark}
\newtheorem{remark}[theorem]{Remark~}
\newtheorem{example}[theorem]{Example~}

\theoremstyle{definition}
\newtheorem{definition}[theorem]{Definition~}

\newcommand{\abs}[1]{\left\vert#1\right\vert}

\newcommand{\sV}{{\mathcal V}}

\newcommand{\BC}{\mathbb C}
\newcommand{\BD}{\mathbb D}
\newcommand{\BF}{\mathbb F}

\newcommand{\BL}{\mathbb L}

\newcommand{\BQ}{\mathbb Q}
\newcommand{\BR}{\mathbb R}

\newcommand{\BZ}{\mathbb Z}

\newcommand{\md}{\mathrm{mod}}

\newcommand{\Co}{\operatorname{Co}}

\newcommand{\NS}{\operatorname{NS}}

\newcommand{\Pic}{\operatorname{Pic}}

\newcommand{\Fr}{\operatorname{Fr}}

\newcommand{\II}{\operatorname{II}}

\newcommand{\GL}{\operatorname{GL}}

\newcommand{\Amp}{\operatorname{Amp}}

\newcommand{\Hom}{\operatorname{Hom}}

\newcommand{\Span}{\mathrm{Span}}

\newcommand{\sgn}{\operatorname{sgn}}

\newcommand{\diag}{\operatorname{diag}}

\newcommand{\Aut}{\operatorname{Aut}}

\newcommand{\rk}{\operatorname{rank}}

\newcommand{\chr}{\operatorname{char}}
\newcommand{\rank}{\operatorname{rank}}

\begin{document}

\title[]{Finite Groups of Symplectic Automorphisms of Supersingular K3 surfaces in Odd Characteristics}%
\author{Bin Wang, Zhiwei Zheng}%
\address{}%
\email{}%

\newcommand{\Addresses}{{
		\bigskip
		\footnotesize
		
		B.~Wang, \textsc{The Chinese University of Hong Kong, Hong Kong, China}\par\nopagebreak
		\textit{Email address}: \texttt{binwang001@cuhk.edu.hk}
		
		\medskip
		
		Z.~Zheng, \textsc{Tsinghua University, Beijing, China}\par\nopagebreak
		\textit{Email address}: \texttt{zhengzhiwei@mail.tsinghua.edu.cn}
}}

\thanks{}%

\begin{abstract}
In 2009, Dolgachev--Keum showed that finite groups of tame symplectic automorphisms of K3 surfaces in positive characteristics are subgroups of the Mathieu group of degree $23$. In this paper, we utilize lattice-theoretic methods to investigate symplectic actions of finite groups $G$ on K3 surfaces in odd characteristics. For supersingular K3 surfaces with Artin invariants at least two, we develop a new machinery called $p$-root pairs to constrain possible symplectic finite group actions (without the assumption of tameness). The concept of $p$-root pair is closely related to root systems and Weyl groups. In particular, we provide alternative proof for many results by Dolgachev--Keum and give an upper bound for the exponent of $p$ in $|G|$.  
\end{abstract}
\maketitle
\setcounter{tocdepth}{1}
\tableofcontents

\section{Introduction}
K3 surfaces play an important role in many areas of mathematics and the study of their automorphism groups has a long and rich history. Nikulin \cite{nikulin1979finite} first systematically studied automorphisms of K3 surfaces via lattice theory. After that, in his pioneering work \cite{Muk88}, Mukai proved (using the so-called Mathieu representations) that a finite group acts faithfully and symplectically on a complex K3 surface if and only if it is a subgroup of the Mathieu group $M_{23}$ with the induced action on the set of $24$ points having at least $5$ orbits. The problem of classifying (symplectic) automorphism groups of K3 surfaces over fields of positive characteristics turns out to be more complicated. In \cite{DK09}, Dolgachev and Keum extended Mukai's methods to prove that if a finite group $G$ can act symplectically and faithfully on a K3 surface over an algebraically closed field of characteristic $p>0$ coprime to $|G|$, then $G$ is isomorphic to a subgroup of $M_{23}$. They also showed that if $p\ge 13$, then any symplectic and faithful action of a finite group $G$ on a K3 surface in characteristic $p$ must be tame, namely, $\gcd(|G|, p)=1$. Using the induced action on N\'eron--Severi lattice, Ohashi--Sch\"utt \cite{OS24} recently provided a complete classification of finite groups that can act on supersingular K3 surfaces with Artin invariant $1$. Combining with Dogalchev--Keum \cite{DK09}, this provides, for any prime $p$, a classification of finite groups which admit a tame symplectic action on a K3 surface in characteristic $p$.

In this paper, we investigate finite groups acting symplectically on K3 surfaces in odd characteristic. Our approach is also based on the induced action on N\'eron--Severi lattice. A novelty of this work is that we uniformly deal with all K3 surfaces in odd characteristics using Leech pairs (see Definition \ref{definition: leech pair}). In particular, we develop a new machinery called $p$-root pairs to link the wild symplectic actions on supersingular K3 surfaces of arbitrary Artin invariants with Weyl group actions on root lattices. By leveraging this relation, we obtain constraints of $p$-parts of a finite group acting symplectically and faithfully on a supersingular K3 surface in odd characteristic $p$.

Before stating our main results, we fix some terminologies. Throughout this paper, let $p$ be an odd prime, and $k$ be any algebraically closed field with $\chr(k)=p$. Let $X$ be a K3 surface in characteristic $p$ (which means that it is a K3 surface over $k$ in this paper). We denote the N\'eron--Severi lattice of $X$ by $\NS(X)$ with rank $\rho(X)$ (which is the Picard number of $X$). It is known that $\NS(X)$ is an even lattice with signature $(1,\rho(X)-1)$. The discriminant group of the lattice $\NS(X)$ is denoted by $A_{\NS(X)}$. We denote the automorphism group of $X$ by $\Aut(X)$. An automorphism $g$ of $X$ is symplectic if and only if $g$ acts identically on $H^{0}(X,\omega_{X})$. Here $\omega_X$ is the canonical line bundle on $X$. We denote by $\Aut^{s}(X)$ the group of symplectic automorphisms of $X$. Given a finite subgroup $G\le \Aut^{s}(X)$, there is an induced action of $G$ on the N\'eron--Severi lattice $\NS(X)$. 

In general, if we have an action of a finite group $G$ on a lattice $L$, we denote by $L^G$ the invariant sublattice, and by $S_G(L)$ the orthogonal complement of $L^G$ in $L$.  We call the lattice $S_G(L)$ the covariant lattice associated with the group action.   We put $S_{G}(X)=S_G(\NS(X))$ for short. This lattice is always rootless because any finite group action on $X$ fixes certain ample class (see discussion in \S\ref{section: supersingular K3}). The pair $(G, S_G(X))$ is a so-called Leech pair, see Definition \ref{definition: leech pair}. We denote by $\BL$ the Leech lattice, and by $\Co_0$ the Conway group (see \S\ref{subsection: leech} for more details.). One key point of our approach is that a Leech pair $(G,S)$ is a subpair of $(\Co_0, \BL)$ if $\rank(S)+\ell(A_S)\le 24$, see Proposition \ref{prop:emb leech}. Here $A_S$ represents the discriminant group of $S$, and for a finitely generated abelian group $A$, $\ell(A)$ is the minimal number of generators.  H\"ohn--Mason \cite{HM16} classified all possible sublattices of $\BL$ which are covariant lattices for certain finite group actions on $\BL$ (there are $290$ cases). This is the main lattice-theoretic input for our approach.

Suppose a group $G$ acts faithfully on a lattice $L$ such that $L^G=\{0\}$ and $G$ acts trivially on $A_L$. Then we denote by $\widetilde{G}$ the group of all automorphisms of $L$ which act trivially on $A_L$. We call $\widetilde{G}$ and $(\widetilde{G}, L)$ the saturation of $G$ and $(G,L)$ respectively. We call $(G,L)$ saturated if $G=\widetilde{G}$.

The K3 surface $X/k$ is said to be (Shioda-)supersingular if $\rho(X)$ reaches the maximal possibility $22$. When $X$ is supersingular, $$A_{\NS(X)}\cong (\BZ/p\BZ)^{2\sigma}, 1\le\sigma\le 10.$$ The number $\sigma$ is called the Artin invariant of $X$. For each $1\le\sigma\le 10$, the N\'eron--Severi lattices of all supersingular K3 surfaces with Artin invariant $\sigma$ are isomorphic, and we denote it by $N_{p,\sigma}$, see Proposition-Definition \ref{prop:R-S lattice}. A supersingular K3 surface is called superspecial if its Artin invariant is 1. It is known that there exists a unique superspecial K3 surface over $k$ up to isomorphism. If a K3 surface in characteristic $p$ is not supersingular, then it is of finite height with height $1\le h\le 10$, see discussions in \S\ref{section: tame}.

Employing the lattice-theoretic method, we prove the following proposition. 
\begin{proposition}
\label{proposition: criterion}
Fix any odd prime $p$. Let $(G,S)$ be a Leech pair with $\rk(S)\leq 21$. The following are equivalent:
\begin{enumerate}
\item for each supersingular K3 surface $X/k$ in characteristic $p$ with Artin invariant $\sigma$, there is an embedding $G\hookrightarrow \Aut^{s}(X)$, with $(G,S_{G}(X))\cong (G,S)$;
\item there exists a supersingular K3 surface $X/k$ in characteristic $p$ with Artin invariant $\sigma$, such that there is an embedding $G\hookrightarrow \Aut^{s}(X)$, with $(G,S_{G}(X))\cong (G,S)$;
\item there exists a primitive embedding $S \hookrightarrow N_{p,\sigma}$.
	\end{enumerate}
\end{proposition}

The superspecial case of Proposition \ref{proposition: criterion} was obtained by Ohashi--Sch\"utt \cite{OS24}. We collect all saturated Leech pairs $(G,S)$ that admit a primitive embedding into $N_{p,1}$ (for all odd prime $p$) in Table \ref{table: sublist of HM}. This table was essentially obtained by Ohashi--Sch\"utt \cite{OS24}. We include it in Appendix \ref{appendix} for the reader's convenience. The calculation is explained in the proof of Proposition \ref{proposition: superspecial list}. Table \ref{table: sublist of HM} will be used in understanding the size of finite groups of symplectic automorphisms of supersingular K3 surfaces with arbitrary Artin invariants, see Proposition \ref{proposition: G_L} and \S\ref{section: application of p-root pair}. 

An action of a finite group $G$ on a K3 surface in characteristic $p$ is called tame if $\gcd(|G|, p)=1$. We prove the following two results about tame symplectic actions.
\begin{proposition}
\label{proposition: tame supersingular}
For a tame symplectic faithful action of a finite group $G$ on a supersingular K3 surface $X$ in characteristic $p\ge 3$ with Artin invariant $\sigma$, we have $\rank(S_G(X))\le 22-2\sigma$. Moreover, the saturated pair $(\widetilde{G},S_G(X))$ is isomorphic to a pair in Table \ref{table: sublist of HM}. Conversely, each pair in Table \ref{table: sublist of HM} (with corresponding prime $p$) is isomorphic to $(G, S_G(X))$ for a tame action of a finite group $G$ on a superspecial K3 surface $X$ in characteristic $p$.
\end{proposition}

\begin{proposition}
\label{proposition: finite height and tame}
For a tame symplectic faithful action of a finite group $G$ on a K3 surface $X$ of finite height $h$ in characteristic $p\ge 3$, we have $\rank(S_G(X))\le 21-2h$, and there exists a primitive embedding of $S_G(X)$ into the K3 lattice.
\end{proposition}

In both situation of Proposition \ref{proposition: tame supersingular} and Proposition \ref{proposition: finite height and tame}, the pair $(G, S_G(X))$ is always a subpair of $(\Co_0, \BL)$ and $\rank(S_G(X))+\ell(A_{S_G(X)})\le 22$. By H\"ohn--Mason's table in \cite{HM16}, a Leech pair $(G,S)$ such that $\rank(S)+\ell(A_S)\le 22$ implies that $G$ is a subgroup of $M_{23}\le \Co_0$. Therefore, Proposition \ref{proposition: tame supersingular} and Proposition \ref{proposition: finite height and tame} provide an alternative proof (in odd characteristic cases) for the significant result by Dolgachev--Keum \cite[Theorem 4.7]{DK09} that any finite group of tame symplectic automorphisms of a K3 surface is a subgroup of $M_{23}$.
See also Ohashi--Sch\"utt \cite[Theorem 1.1]{OS24}.


Let $X$ be supersingular with Artin invariant $\sigma$. We can extend the action of $G$ on $N_{p,\sigma}$ to $N_{p,1}$. The lattice $S_G(N_{p,\sigma})$ is rootless, while $S_{G}(N_{p,1})$ might contain roots. We denote by $R$ the root lattice of $S_{G}(N_{p,1})$. Apparently, $G$ acts on $R$. Let $G_R$ be the image of $G$ in $\Aut(R)$, and $G_L=\mathrm{Ker}(G\to \Aut(R))$. Then we have an exact sequence:
\[
1\rightarrow G_L\rightarrow G\rightarrow G_R\rightarrow 1.
\] 
Further exploring the action of $G_R$ on $R$, we introduce the concept of (pseudo) p-root pairs (see \S\ref{section: p-root pair}). 
\begin{definition}
Let $H$ be a finite group acting faithfully on a root lattice $R$. We call $(H,R)$ a \emph{pseudo $p$-root pair} if the lattice $R_H^\#$ generated by $pR$ and  $\{g\alpha-\alpha\big{|} g\in H,\alpha\in R\}$ contains no roots. We call $(H,R)$ a \emph{$p$-root pair} if it is a pseudo $p$-root pair, and in addition $R^H=\{0\}$.
\end{definition}
We summarize most of our results on $p$-root pairs and their applications to supersingular K3 surfaces in \S\ref{section: p-root pair} and \S\ref{section: application of p-root pair} as follows:
\begin{theorem}
\label{theorem: main3}
    Let $(H,R)$ be a $p$-root pair, then:
      \begin{enumerate}
        \item If $R=A_m$, then $m+1$ is a power of $p$ (Proposition \ref{prop: type A}).
        \item $R$ cannot be of type $D_m, m\ge 5$ (Proposition \ref{proposition: no p-root pair of type D}). If $R=D_4$, $p=3$ (Subsection \ref{sec:D4}).
        \item If $p\ge 7$, $R$ cannot be of type $E$ (Proposition \ref{proposition: p>=7 no type E}).
        \item $H\cap W(R)$ is always a $p$-group, here $W(R)$ denotes the Weyl group of $R$ (Proposition \ref{prop: intersection p group}).
    \end{enumerate}
\end{theorem}

\begin{proposition}
\label{proposition: geometry gives p-root pair}
    Let $X$ be a supersingular K3 surface in characteristic $p$ with Artin invariant $\sigma$. Let $G$ be a finite group acting faithfully and symplectically on $X$. 
    \begin{enumerate}
   \item If $G$ is tame, then $G_R$ is trivial and $G=G_L$. 
   \item $(G_R, R)$ is a $p$-root pair.  
    \end{enumerate}
\end{proposition}

We here point out several interesting corollaries. For $p\ge 13$, we deduce that a finite group acting faithfully and symplectically on a supersingular K3 surface is always tame, see Theorem \ref{thm:new DK}.  This provides an alternative proof of \cite[Theorem 2.1]{DK09} for supersingular K3 surfaces. 

We also deduce that if $p=11$, then the exponent of $p$ in the order of $G$ is at most $1$. This corollary also follows from  Dolgachev--Keum \cite[Theorem 1.1]{DK09J}. See Proposition \ref{proposition: precise estimate of 11-order}. 

Moreover, combining both analysis on $G_L$ and $G_R$, we can also obtain an upper bound  on the exponent of $p$ in $|G|$ for a finite group $G$ of symplectic automorphisms of a K3 surface in characteristic $p=3,5,7$. See Proposition \ref{proposition: estimate of p-order}.

Now let us recall how our approach is related with earlier works. Kond\=o \cite{Kon98} gave another proof of Mukai's results which is lattice-theoretic and motivates our proof here. Ogus \cite{Ogu79} (respectively Bragg--Lieblich's unpublished work \cite{BL18}) proved the crystalline Torelli theorem for supersingular K3 surfaces in char $p\geq 5$ (respectively $p\geq 3$), which makes the lattice theory a powerful tool in the study of $\Aut^{s}(X)$. Nygaard \cite{Nyg80} and Jang \cite{Jan16} studied the induced action of an automorphism on the Witt vector cohomology and the discriminant group of the N\'eron--Severi group in odd characteristics respectively, which enables us to obtain Leech pairs. The key input on lattice-theoretic side is from H\"ohn--Mason \cite{HM16}, Jang \cite{Jan16} and also inspired by methods developed in Laza--Zheng \cite{laza2022symplectic}.

\noindent\textbf{Structure of the paper}:
In \S\ref{sec:lat}, we collect lattice-theoretic results that will be used in our paper. A key result is Proposition \ref{prop:emb leech}. In \S\ref{section: supersingular K3} we introduce supersingular K3 surfaces and recall the crystalline Torelli theorem by Ogus, Bragg and Lieblich. In particular, we give the motivation for the concept of $p$-root pairs. In \S\ref{section: tame}, we prove Proposition \ref{proposition: tame supersingular}, \ref{proposition: finite height and tame} and discuss their relation with previous works. In \S\ref{section: p-root pair}, we developed the new machinery of (pseudo) $p$-root pairs, and in \S\ref{section: application of p-root pair} we apply the new machinery to reprove an important result by Dolgachev--Keum (Theorem \ref{thm:new DK}) in supersingular cases and give upper bounds on the wild degree of finite groups of symplectic automophisms of K3 surfaces in characteristics $p=3,5,7,11$. In Appendix A, we include the full list of H\"ohn--Mason pairs admitting primitive embeddings in $N_{p,1}$.

\noindent\textbf{Acknowledgement}: Bin Wang is supported by General Research Fund 14307022 and Early Career Scheme 24307121 of Michael McBreen. Zhiwei Zheng is partially supported by NSFC 12301058. We thank Prof. Looijenga for helpful discussions, especially about $p$-root pairs. We thank Matthias Sch\"utt for many helpful comments after this paper being posted on arXiv.

\section{Lattice-Theoretical Results}
\label{sec:lat}
In this section, we introduce some terminologies and results from lattice theory that will be used. We refer to \cite{Nik79}, \cite[Chapter 15]{conway1999spherepackings}, \cite[Appendix A]{laza2022symplectic} for more details.
\subsection{Integral Lattices and $p$-Adic Lattices} 
Let $R$ be a PID (principal integral domain) with fractional field $K$. An $R$-lattice is a free $R$-module $L$ of finite rank endowed with a non-degenerate symmetric bilinear form $\phi:L\times L\rightarrow R$, which can be extended to a non-degenerate symmetric $K$-valued bilinear form on $L_K\coloneqq L\otimes K$. Denote by $L^{\vee}\coloneqq \Hom_R(L,R)$ the dual of $L$. We have natural embeddings $L\hookrightarrow L^{\vee}\hookrightarrow L_K$ with the first map sending $x$ to $\phi(x, \cdot)$. The discriminant group of $L$ is by definition the quotient $A_{L}:=L^{\vee}/L$, which is a finite abelian group. 
 
\begin{definition}
Given two ($R$-)lattices $L_{1}$ and $L_{2}$, and an embedding $i:L_{1}\hookrightarrow L_{2}$
	\begin{itemize}
		\item[(1)] $i$ is said to be primitive, if $L_{2}/L_{1}$ is a torsion-free $R$-module,
		\item[(2)] $i$ is said to be a saturation of $L_1$, if $\rk L_{1}=\rk L_{2}$.
	\end{itemize}
\end{definition}

A lattice over $R=\BZ$ is called an integral lattice; for a prime $p$, a lattice over $R=\BZ_p$ is called a $p$-adic lattice.
An integral (or $2$-adic) lattice $L$ is said to be even if for any $v\in L$, its norm $\phi(v,v)$ is even. Throughout this paper, we only encounter even lattices while over $\BZ$ or $\BZ_2$. An integral lattice $L$ is said to be unimodular if $A_{L}=0$.

\subsection{Finite Quadratic Forms}
We give a brief summary of finite quadratic forms following Nikulin \cite[$\S$ 1]{Nik79}.
\begin{definition}
	\begin{itemize}
		\item[(I)] A finite symmetric bilinear form consists of a finite abelian group $A$ and a symmetric $\BZ$-bilinear map $b:A\times A\rightarrow \BQ/\BZ$.
		\item[(II)] A finite quadratic form consists of a finite abelian group $A$ and a function $q:A\rightarrow \BQ/2\BZ$ satisfying the following properties:
		\begin{enumerate}
			\item[(1)] $q(na)=n^{2}q(a)$, for $n\in \BZ, a\in A$; and
			\item[(2)] if $b\colon A\times A\to \BQ/\BZ$ is defined by $b(a,b):=\frac{1}{2}(q(a+b)-q(a)-q(b))\in \BQ/\BZ$, then $(A, b)$ is a finite symmetric bilinear form. 
		\end{enumerate}
	We call $b$ the bilinear form of $q$. We say $q$ is non-degenerate if $b$ is.
	\end{itemize}

\end{definition}

Let $(L,\phi)$ be an even integral lattice. We define $q_L$ to be a function on $A_L$ sending $[x]$ to $[\phi(x,x)]$ for $x\in L^{\vee}$. The pair $(A_L, q_L)$ is called the discriminant form associated with $L$. It is straightforward to check that $(A_{L},q_{L})$ is a non-degenerate finite quadratic form. We have the following decomposition:
\begin{equation}\label{eq:decomp}
\BQ/\BZ\cong \oplus_{p}\BQ_{p}/\BZ_{p}, \BQ/2\BZ\cong (\oplus_{p\neq 2}\BQ_{p}/\BZ_{p})\oplus \BQ_{2}/2\BZ_{2}
\end{equation}
\begin{proposition}\cite[\S 1, $2^{\circ}$]{Nik79}
	Let $(A,q)$ be a finite quadratic form. We decompose $A$ into a direct sum of its $p$-part $$A=\oplus_{p} A_{p}$$ 
	and $q_{p}=q|_{A_{p}}$. Then $(A,q)\cong \oplus_{p}(A_{p},q_{p})$. Via the decomposition \eqref{eq:decomp}, $q_{p}$ takes values in $\BQ_{p}/\BZ_{p}$ for $p\neq2 $, and $q_{2}$ takes values in $\BQ_{2}/2\BZ_{2}$. 
\end{proposition}

For any finite abelian group $A$, we denote by $\ell(A)$ the minimal number of generators of $A$, and denote $\ell_p(A)=\ell(A_p)$ for any prime $p$.

For the discriminant form $q_L$ of $L$, we define its signature $\tau(q_L)$ as the signature of $L$ modulo $8$. By \cite[Theorem 1.3.3]{Nik79}, $\tau(q_L)$ only depends on the finite quadratic form $q_L$.

\subsection{Conway--Sloane Symbol}
We take $R=\BZ_{p}$. 

Following Nikulin \cite[\S 1, $8^{\circ}$]{Nik79}, we denote the one-dimensional $\BZ_{p}$-lattice $(\theta p^{k})$ by $K^{(p)}_{\theta,k}$, where $\theta\in\BZ_{p}^{\times}$. And when $p=2$, we put:
\[
U^{(2)}_{k}=\begin{pmatrix}
0&2^{k}\\
2^{k}&0
\end{pmatrix},V^{(2)}_{k}=\begin{pmatrix}
2^{k+1}&2^{k}\\
2^{k}&2^{k+1}
\end{pmatrix}
\]
We denote the semigroup of isomorphism classes of p-adic lattices with respect to "$\oplus$" by $\mathrm{Qu}(\BZ_{p})$.
\begin{proposition}\cite[Proposition 1.8.3]{Nik79}\label{prop:decomp of p-adic lat}
	\begin{enumerate}
		\item if $p\neq 2$, $\mathrm{Qu}(\BZ_{p})$ is generated by $K^{(p)}_{\theta,k}$;
		\item if $p=2$, $\mathrm{Qu}(\BZ_{2})$ is generated by $K^{(2)}_{\theta,k}$, $U^{(2)}_{k}$ and $V^{(2)}_{k}$.
	\end{enumerate}
\end{proposition}
\begin{remark}
	Given a $p$-adic lattice, its expression as a direct sum of $K^{(p)}_{\theta,k}$ (and  $U^{(2)}_{k}$,  $V^{(2)}_{k}$ for $p=2$) may not be unique, see \cite[\S1, $8^{\circ}$]{Nik79}.
\end{remark}
When $p$ is odd, let $(L,\phi)$ be a p-adic lattice. Then by Proposition \ref{prop:decomp of p-adic lat}, there exists a $\BZ_{p}$-basis of $L$, such that with respect to this basis $\phi$ is a diagonal matrix of the following form:
\[
\diag(p^{s_{1}}M_{1},\cdots,p^{s_{t}}M_{t}).
\]
Here $0\le s_1<s_2<\cdots<s_t$, each $M_{i}$ is a diagonal matrix of rank $r_{i}$, with diagonal terms in $\BZ_p^\times$. 

We define the Conway--Sloane symbol of $L$ as $\prod\limits_{s_i>0}(p^{s_{i}})^{\epsilon_{i}r_{i}}$
where 
\begin{equation*}
\epsilon_{i}=
\begin{cases}
+, \text{if} \det(M_{i}) \text{ is a square in } \BZ_p, \\
-, \text{otherwise}.
\end{cases}
\end{equation*}

For $p=2$, let $(L,\phi)$ be a $2$-adic 
even lattice. The Conway--Sloane symbol  is more complicated. By Proposition \ref{prop:decomp of p-adic lat} and \cite[Proposition 1.8.2]{nikulin1979finite}, there exists a $\BZ_{2}$-basis of $L$, such that with respect to this basis $\phi$ is a block diagonal matrix of the following form:
\begin{equation*}
\diag(2^{s_{1}}M_{1},\cdots, 2^{s_{t}}M_{t}).
\end{equation*}
Here $M_i$ is either a diagonal matrix with diagonal terms in $\BZ_2^\times$, or a block diagonal matrix with diagonal terms $U_0^{(2)}$ and $V_0^{(2)}$. In the former case we call $M_i$ of odd type, and for the latter case we call $M_i$ of even type. Suppose $M_i$ has rank $r_i$, which does not depend on the choices of basis. The Conway--Sloane symbol is defined as $\prod\limits_{s_i>0} {(2^{s_i})^{\epsilon_i r_i}_{t_i}}$ where
\begin{equation*}
\epsilon_{i}=
\begin{cases}
+, \text{if} \det(M_{i})\equiv \pm 1\; (\md\; 8), \\
-, \text{otherwise},
\end{cases}
\end{equation*}
and
\begin{equation*}
t_i=
\begin{cases}
\mathrm{Tr}(M_i) \;(\md\; 8), \text{ if } M_i \text{ is of odd type}, \\
\II, \text{otherwise}.
\end{cases}
\end{equation*}

For an even integral lattice $L$, putting together the Conway--Sloane symbol of $L$ as a $p$-adic lattice for all possible $p$, we obtain the Conway--Sloane symbol of $L$.

If two even integral lattices $L$ and $L'$ have the same Conway--Sloane symbol, then the discriminant forms $(A_{L},q_{L})$ and $(A_{L'},q_{L'})$ are isomorphic. In other words, Conway--Sloane symbol determines the discriminant form up to isomorphism.

\subsection{Niemeier Lattices and Leech Pairs}
\label{subsection: leech}
From now on, lattices are always integral if not stated otherwise.
\begin{definition}
	A Niemeier lattice is an even unimodular  negative lattice with rank 24. There are 24 Niemeier lattices. Among the 24 Niemeier lattices, there is a unique one called the Leech lattice, denoted by $\BL$, which is rootless (namely, $\BL$ does not contain elements with norm -2).
\end{definition}

The Leech lattice will play an important role in our study of finite symplectic automorphism groups of supersingular K3 surfaces.
\begin{definition}
	We denote the automorphism group of the Leech lattice $\BL$ by $\Co_{0}$. And we put $\Co_{1}=\Co_{0}/\{\pm 1\}$. They are called Conway groups. The Conway group $\Co_{1}$ is a sporadic simple group.
\end{definition}

\begin{definition}
\label{definition: leech pair}\cite[Definition 3.3]{laza2022symplectic}
For a finite group $G$ acting faithfully on a negative definite even lattice $S$, the pair $(G, S)$ is called a Leech pair, if $S$ is rootless, the induced action of $G$ on $A_S$ is trivial, and the action has no non-zero fixed vectors. 
\end{definition}



The following criterion is crucial for us.
\begin{proposition}
\label{prop:emb leech}
	Given a Leech pair $(G,L)$ with $\rk(L)+\ell(A_{L})\leq 24$, then $(G,L)$ is a subpair of $(\Co_0, \BL)$. In particular, $G$ is a subgroup of $\Co_0$.
\end{proposition}

This proposition is originally from \cite[appendix]{gaberdiel2012symmetries}, see also \cite[Corollary 4.19]{marquand2025Ogrady10} and \cite{zheng2025leechlemma}.

\subsection{Some Lemmas on Lattice Embedding}
We here collect some lemmas which will be frequently used in our later discussion.
\begin{lemma}
\label{lemma: sum}
Let $M\hookrightarrow N$ be a primitive embedding of two even lattices $N, M$, then $\rank(M)+\ell_q(A_M)\le \rank(N)+\ell_q(A_N)$ for any prime number $q$. Moreover, $\rank(M)+\ell(A_M)\le \rank(N)+\ell(A_N)$.
\end{lemma}
\begin{proof}
Take a primitive embedding of $N$ into an even unimodular lattice $L$. Denote by $N_L^{\perp}, M_L^{\perp}$ the orthogonal complements of $N,M$ in $L$ respectively. Then we have a primitive embedding $N^{\perp}_L\hookrightarrow M^{\perp}_L$. Let $K$ be the orthogonal complement of $N^{\perp}_L$ in $M^{\perp}_L$. Then we have a saturation $N^{\perp}_L\oplus K \hookrightarrow M^{\perp}_L$. This implies that:
$$\ell_q(A_{M^{\perp}_L}) \le \ell_q(A_{N^{\perp}_L}\oplus A_K).$$
Combining with $A_{M^{\perp}_L}\cong A_M$, $A_{N^{\perp}_L}\cong A_N$ and $\ell(A_K)\le \rank(K)=\rank(N)-\rank(M)$ we conclude $\rank(M)+\ell_q(A_M)\le \rank(N)+\ell_q(A_N)$. Replace $l_q$ by $l$, the above argument still holds and we conclude $\rank(M)+\ell(A_M)\le \rank(N)+\ell(A_N)$.
\end{proof}

Let $G$ be a group acting on a lattice $L$. Recall that $L^{G}$ is the $G$-invariant sublattice of $L$ and $S_{G}(L)$ is the orthogonal complement of $L^G$ in $L$, which is called the covariant lattice of the action.

\begin{lemma}
\label{lemma:trivial action}
Suppose a group $G$ acts on an even lattice $L$, such that the induced action of $G$ on $A_{L}$ is trivial. Then $G$ acts trivially on $A_{S_{G}(L)}$.
\end{lemma}
\begin{proof}
	Since $S_{G}(L)\hookrightarrow L$ is a primitive embedding, we have a surjection $L^{\vee}\twoheadrightarrow (S_{G}(L))^{\vee}$. In fact, we have:
	\[
	S_{G}(L)\hookrightarrow L\hookrightarrow L^{\vee}\twoheadrightarrow (S_{G}(L))^{\vee}
	\] 
	Given an element in $A_{S_{G}(L)}$, we can choose a representative $x\in L^{\vee}$. Since $G$ acts trivially on $A_{L}$, we have $gx-x\in L$ for any $g\in G$. Let $\phi$ be the bilinear form on $L$. For any $y\in L^{G}$, we have $\phi(gx-x,y)=0$, which implies that $gx-x\in S_{G}(L)$, i.e. $g[x]=[x]\in A_{S_{G}(L)}$.
\end{proof}

\begin{lemma}
\label{lemma: extending group action}
Suppose $M\hookrightarrow L$ is an embedding of even lattices. Suppose $G< \Aut(M)$ is a finite subgroup such that the induced action on $A_M$ is trivial. Then there exists an action of $G$ on $L$ extending the action of $G$ on $M$, such that any vector of $L$ orthogonal to $M$ is fixed by the action. Moreover, the induced action of $G$ on $A_L$ is trivial.
\end{lemma}
\begin{proof}
We may assume $M$ is primitive in $L$ via replacing it by its saturation in $L$. Take a primitive embedding of $L$ into a unimodular even lattice $N$. The action of $G$ on $M$ extends to an action on $N$ that restricts to identity on the orthogonal complement of $M$ in $N$. Then the restriction of this action to $L$ gives the required action.
\end{proof}

\subsection{Supersingular K3 Lattices}
\label{section: K3 lat}

We begin with the following proposition, in which the lattice encodes the properties of the N\'eron--Severi lattice of a supersingular K3 surface.
\begin{proposition-definition}\cite[\S 1]{RS81}\label{prop:R-S lattice}
For an odd prime $p$, and $1\leq \sigma\leq 10$, there exists a unique lattice, denoted as $N_{p,\sigma}$, with the following properties:
	\begin{itemize}
		\item[(1)] The lattice $N_{p,\sigma}$ is even;
		\item[(2)] The signature of $N_{p,\sigma}$ is (1,21);
		\item[(3)] The discriminant group $A_{N_{p,\sigma}}$ is isomorphic to $(\BZ/p\BZ)^{2\sigma}$.
	\end{itemize}
    And we call $N_{p,\sigma}$ the supersingular K3 lattice in characteristic $p$ and with Artin invariant $\sigma$.
\end{proposition-definition}

In the following proposition, we recall the calculation of the Conway--Sloane symbol of a supersingular K3 lattice.
\begin{proposition}\label{prop:discr form of ss K3 lattice}
	Let $N_{p,\sigma}$ be a supersingular K3 lattice as above. Then we have:
	\begin{enumerate}
\item If $p\equiv 3 \;(\md\;4)$, then:
		\[
	q_{N_{p,\sigma}}=
 \begin{cases}
p^{+2\sigma}, \text{if}\;\sigma\;\text{is odd}, \\
p^{-2\sigma}, \text{if}\;\sigma\;\text{is even}.
\end{cases}
\]
\item If $p\equiv 1 \;(\md\; 4)$, then $q_{N_{p,\sigma}}=p^{-2\sigma}$.
	\end{enumerate}
\end{proposition}
\begin{proof}
Since $N_{p,\sigma}$ has signature $(1,21)$, we have $\tau(q_{N_{p,\sigma}})\equiv 4\; (\md\; 8)$. Since $A_{N_{p,\sigma}}\cong (\BZ/p\BZ)^{2\sigma}$, we know that the discriminant form of $N_{p,\sigma}$ will be $p^{+2\sigma}$ or $p^{-2\sigma}$. Then the proposition follows from the following:
	\[
	\begin{cases}
\tau(p^{+1})\equiv 1-p \;(\md\; 8),\\
\tau(p^{-1})\equiv 5-p \;(\md\; 8).
	\end{cases}
	\]
	(See \cite[Proposition 1.11.2]{Nik79} or \cite[\S 2.1]{NT98}).
\end{proof}
\begin{corollary}
	There is no maximal isotropic subspace of $(A_{N_{p,\sigma}}, q_{N_{p,\sigma}})$ of dimension $\sigma$.
\end{corollary}
\begin{proof}
	We argue by contradiction. If such a maximal isotropic subspace exists, then we have:
	\[
	q_{N_{p,\sigma}}=
	\begin{cases}
p^{+2\sigma}, \text{if}\; (-1)^{\sigma}\;\text{is a quadratic residue}, \\
p^{-2\sigma}, \text{if}\; (-1)^{\sigma}\;\text{is not a quadratic residue}.
	\end{cases}
	\]
	Hence the corollary follows from the Proposition \ref{prop:discr form of ss K3 lattice}.
\end{proof}

\section{Supersingular K3 Surfaces}
\label{section: supersingular K3}
In this section, we first collect some results on supersingular K3 surfaces. In \S\ref{subsec:superspecial}, we specialize to superspecial K3 surfaces. 

\subsection{N\'eron--Severi Lattices}
In this subsection, we focus on N\'eron--Severi lattices of supersingular K3 surfaces in odd characteristics. Let $X$ be a K3 surface in characteristic $p$, then we have:
\[
\Pic(X)\cong \NS(X)\cong\BZ^{\oplus\rho(X)}, 1\leq\rho(X)\leq 22.
\]
It is known that $\NS(X)$ is  an even lattice of signature $(1,\rho(X)-1)$, see \cite[\S 17.2]{Huy16}. The K3 surface $X$ is called (Shioda-)supersingular if $\rho(X)=22$. 
\begin{theorem}\cite{Art74, RS81}
	Let $X$ be a supersingular K3 surface. Then there exists an integer $1\leq \sigma\leq 10$, the Artin invariant, such that:
	\[
	A_{\NS(X)}\cong (\BZ/p\BZ)^{2\sigma}.
	\]
\end{theorem}

By Proposition \ref{prop:R-S lattice}, for a supersingular K3 surface $X$ with Artin invariant $\sigma$, we have $\NS(X)\cong N_{p,\sigma}$.

\begin{definition}\cite[Definition 3.19]{Ogu79}\cite[Definition 4.14]{Lie13}
	Let $\sigma\geq 1$ be an integer and $V$ be a $2\sigma$-dimensional $\BF_{p}$-vector
	space endowed with a nondegenerate and nonneutral quadratic form. A subspace $K\subset V\otimes_{\BF_{p}} k$ is called characteristic if
	\begin{itemize}
		\item[(a)] $K$ is totally isotropic of dimension $\sigma$, i.e., all the vectors in $K$ are isotropic, and
		\item[(b)] $K+(Id\otimes \Fr_{k})(K)$ is of dimension $\sigma+1$. Here $\Fr_k: k\rightarrow k$ is the Frobenius map.
	\end{itemize}
\end{definition}

\subsection{Crystalline Torelli Theorem for Supersingular K3 Surfaces}
We review the crystalline Torelli theorem for marked supersingular $K3$ surfaces.  

Let $X$ be a supersingular K3 surface with Artin invariant  $\sigma$. Next we write $N=N_{p, \sigma}$ for short. An $N$-marking of $X$ is a lattice isomorphism $\phi: N \rightarrow \NS(X)$. Let $(X,\phi), (Y, \psi)$ be two $N$-marked K3 surfaces. An isomorphism between $(X,\phi)$ and $(Y,\psi)$ is an isomorphism $\theta: X\rightarrow Y$ such that $\phi=\theta^{*}\circ\psi$. We define:
$$
\ker(X,\phi):=\ker(c_{k}:N\otimes k\rightarrow H^{2}_{dR}(X/k)).
$$
Here $c_{k}$ is the composition of $N\otimes k\rightarrow\NS(X)\otimes k$ and $\NS(X)\otimes k\rightarrow H^{2}_{crys}(X/W)\otimes k\cong H^{2}_{dR}(X/k)$. The second map is induced by the crystalline Chern class map. Ogus \cite[Proposition 4.3]{Ogu79} showed that $\ker(X,\phi)$ is a subspace of $(pN^{\vee}/pN) \otimes_{\BF_{p}}k$. We then put $K(X,\phi)\coloneqq (id\otimes\Fr_{k})^{-1}(\ker(X,\phi))$. 
\begin{proposition}\cite[\S 1.8]{Ogu83}\label{prop: char of N0}
	Both $\ker(X,\phi)$ and $K(X,\phi)$ are characteristic subspaces of $(pN^{\vee}/pN) \otimes_{\BF_{p}}k$.
\end{proposition}

 We denote the set of roots in $N$ by $\Delta$. A line bundle $L$ on a K3 surface is ample if and only if $L\cdot L>0$, $L\cdot E>0$ for all effective divisors $E$ with $E\cdot E=-2$ and $h^0(L)\ne 0$. Denote by $\Amp(X)$ the ample cone of $X$. Then $\phi^{-1}(\Amp(X))$ is contained in a unique connected component of
$
\{x\in N\otimes \BR|x\cdot x>0,x\cdot r\neq 0, \forall r\in\Delta\}
$. For more details of the above statements, see \cite[\S 1.10]{Ogu83}.
Now we can state the following version of the crystalline Torelli theorem using ample cones for $N$-marked supersingular K3 surfaces.
\begin{theorem}\cite{Ogu83,BL18}\label{thm:marked crys torelli}
	Let $(X,\phi)$ and $(Y,\psi)$ be two $N$-marked supersingular K3 surfaces. If 
$$(K(X, \phi), \phi^{-1}(\Amp(X))=(K(Y, \psi), \psi^{-1}(\Amp(Y)),$$	
then there is a unique isomorphism $f:(X,\phi)\rightarrow (Y,\psi)$. 
\end{theorem}

The theorem we state here appears as \cite[Theorem II$''$]{Ogu83} for $p\geq 5$. It is equivalent to the crystalline Torelli theorem \cite[Theorem II]{Ogu83} which was  proved there for $p\geq 5$ and by Bragg--Lieblich \cite[Theorem 5.1.9]{BL18} for $p\ge 3$. Thus Theorem \ref{thm:marked crys torelli} holds for all odd characteristics.

The following proposition is the key bridge for us to characterize finite groups of symplectic automorphisms of $X$.

\begin{proposition}\label{prop:lat description of sym auto}
For supersingular K3 surface $X$, we have the following isomorphism:
	\[
	\Aut^{s}(X)\cong \{g\in \mathrm{O}^{+}(\NS(X))| g\;\text{is identity on } A_{\NS(X)}\}
	\]
	here $\mathrm{O}^{+}$ means preserving the pairing and ample cone.
\end{proposition}
\begin{proof}
Jang \cite[Proposition 3.1]{Jan16} shows that for any $g\in \Aut(X)$, $g$ acts on $A_{\NS(X)}\otimes_{\BF_{p}} k$ trivially if and only if $g$ is symplectic. See also \cite[Theorem 2.1]{Nyg80} and \cite[Theorem 3.1]{OS24}. By Proposition \ref{prop: char of N0}, $K$ is contained in 
\begin{equation*}
(p\NS(X)^{\vee}/p\NS(X))\otimes k\subset A_{\NS(X)}\otimes k. 
\end{equation*}
Then this proposition is a direct result of Theorem \ref{thm:marked crys torelli}.
\end{proof}
\subsection{A Lattice-Theoretic Interpretation of Symplectic Automorphism Groups}
By \cite[Theorem 1.14.4]{Nik79}, there exists a unique primitive embedding:
\[
N_{p,\sigma}\hookrightarrow\II_{9,41}
\]
up to automorphisms of $\II_{9,41}$. Here the indices $9,41$ are simply chosen large enough to fulfill Nikulin's criterion. We denote by $\BD_{p,\sigma}$ the orthogonal complement of $N_{p,\sigma}$ in $\II_{9,41}$. Then $q_{\small{\BD_{p,\sigma}}}=-q_{N_{p, \sigma}}$ and $\sgn(\BD_{p,\sigma})=(8,20)$. Such a lattice $\BD_{p,\sigma}$ is unique up to isomorphism and the orthogonal complement of any primitive embedding of $\BD_{p,\sigma}$ in $\II_{9,41}$ is isomorphic to $N_{p,\sigma}$. We call $\BD_{p,\sigma}$ the auxiliary lattice. The next proposition will be used to prove Proposition \ref{proposition: criterion}.

\begin{proposition}
\label{prop:lat criterion}
	For a negative definite lattice $S$ with $\rk(S)\leq 21$, the following are equivalent:
	\begin{enumerate}
		\item there exists a primitive embedding $S\hookrightarrow N_{p,\sigma}$;
		\item there exists an embedding $S\oplus\BD_{p,\sigma}\hookrightarrow\II_{9,41}$ such that the image of $S$ is primitive.
	\end{enumerate}
\end{proposition}
\begin{proof}
	Assuming (1), then there is a primitive embedding $S\oplus \BD_{p,\sigma}\hookrightarrow N_{p,\sigma}\oplus\BD_{p,\sigma}$. By definition of $\BD_{p,\sigma}$, we have a saturation $N_{p,\sigma}\oplus\BD_{p,\sigma}\hookrightarrow\II_{9,41}$ with $N_{p,\sigma}$ primitive. The composition of the above two embeddings gives rise to an embedding $S\oplus \BD_{p,\sigma}\hookrightarrow\II_{9,41}$ such that the image of $S$ is primitive.
	
	Conversely, if (2) holds, we have $S\hookrightarrow \BD_{p,\sigma}^{\perp}\cong N_{p,\sigma}$, which is apparently primitive.
\end{proof}

The following lemma is essential for further classifications:
\begin{lemma}
\label{lemma: leech pair from geo}
Let $X$ be a supersingular K3 surface and $G\le \Aut^s(X)$ a finite subgroup. Then $(G, S_G(X))$ is a Leech pair.
\end{lemma}
\begin{proof}
Since $G$ is finite, there exists a $G$-invariant ample class on $X$. Since $\sgn(\NS(X))=(1,21)$, we know that the lattice $S_G(X)$ is negative definite. By Proposition \ref{prop:lat description of sym auto}, $G$ acts trivially on $A_{\NS(X)}$ , then $G$ acts trivially on $A_{S_{G}(X)}$ by Lemma \ref{lemma:trivial action}. Since $S_G(X)$ is orthogonal to the ample class, it cannot contain any effective divisors. Therefore, there are no roots in $S_{G}(X)$.
\end{proof}

We now prove Proposition \ref{proposition: criterion} which is a criterion of Leech pairs arising from symplectic actions of finite groups on supersingular K3 surfaces.

\begin{proof}[Proof of Proposition \ref{proposition: criterion}]
    It is clear that (1) implies (2) and (2) implies (3). Next we show that (3) implies (1). We denote by $H$ the orthogonal complement of $S$ in $N_{p,\sigma}$. By Lemma \ref{lemma: extending group action}, the action of $G$ on $S$ extends to $N_{p,\sigma}$ such that $G$ acts trivially on $A_{N_{p,\sigma}}$ and $H$.
	
	Next, we show the existence of a vector $h\in H$ with positive norm, such that $h_{N_{p,\sigma}}^{\perp}$ contains no roots. Let $C$ be the set of elements in $H$ with positive norm. By \cite[Proposition 1.10]{Ogu83}, the chambers in $(N_{p,\sigma})_{\BR}$ with respect to all roots in $N_{p,\sigma}$ are open. Since $S$ contains no roots, there exists a chamber $\sV$ such that $\sV\cap C_{\BR}\neq \{0\}$. Since $\sV$ is open, we have $\sV\cap C_{\BQ}\neq\{0\}$. By taking an integral scalar of a nonzero element in $\sV\cap H_{\BQ}$, we obtain a primitive vector $h\in H$ with positive norm, such that $h_{N_{p,\sigma}}^{\perp}$ contains no root. By \cite[Proposition 1.10]{Ogu83}, for each supersingular K3 surface $X/k$ with Artin invariant $\sigma$, there exists an isomorphism $f\colon N_{p,\sigma}\cong \NS(X)$, such that $f(h)\in\NS(X)$ is ample. The action on $N_{p,\sigma}$ naturally induces an action of $G$ on $\NS(X)$, such that $f(h)\in\NS(X)$ is fixed and the induced action on $A_{\NS(X)}$ is trivial. Therefore, the characteristic subspace is preserved, and by Proposition \ref{prop:lat description of sym auto}, we obtain an embedding $G\hookrightarrow\Aut^{s}(X)$, such that $(G,S_G(X))\cong (G,S)$.
\end{proof}

\subsection{Superspecial K3 Surfaces}
\label{subsec:superspecial}
In this subsection, we discuss the case of superspecial K3 surfaces, i.e., $\sigma=1$. Maximal finite subgroups of $\Aut^s(X)$ for superspecial K3 surfaces over an algebraically closed field of any characteristic (including $2$) are obtained recently by Ohashi--Sch\"utt \cite{OS24}. Together with \cite{DK09}, they complete the classification of finite tame symplectic actions on all K3 surfaces in positive characteristic. We define H\"ohn--Mason pair as follows:
\begin{definition}
A H\"ohn--Mason pair $(G,S)$ is a Leech pair with $G\le \Co_0$ and $S=S_G(\BL)$, such that $G$ is maximal in the sense that $G$ contains all automorphisms of $S$ acting trivially on $A_S$, i.e., $G$ has no nontrivial saturation.
\end{definition}
Such kind of pairs have been classified by H\"ohn--Mason \cite{HM16} and there are 290 of them. Now we relate finite groups of symplectic automorphisms of a superspecial K3 surface with H\"ohn--Mason pairs. Recall that for any Leech pair $(G,S)$, we let $\widetilde{G}$ be the group of all automorphisms of $S$ acting trivially on $A_S$. The saturation $(\widetilde{G}, S)$ is still a Leech pair.
\begin{proposition}
\label{proposition: superspecial to hohn mason}
Suppose $X$ is a superspecial K3 surface and $G\le \Aut^s(X)$ is finite. Then the pair $(G, S_G(X))$ is a Leech pair and it can be saturated to a H\"ohn--Mason pair $(\widetilde{G}, S_G(X))$.
\end{proposition}
\begin{proof}
By Lemma \ref{lemma: leech pair from geo}, the pair $(G, S_G(X))$ is a Leech pair. We have $\rank(\NS(X))+\ell(A_{\NS(X)})=24$. By Lemma \ref{lemma: sum} we have $\rank(S_G(X))+\ell(A_{S_G(X)})\le 24$. By Proposition \ref{prop:emb leech}, $(G, S_G(X))$ is a subpair of $(\Co_0, \BL)$. Thus $(\widetilde{G}, S_G(X))$ is a H\"ohn--Mason pair.  
\end{proof}

By Proposition \ref{proposition: criterion}, a Leech pair $(G,S)$ arises from a superspecial K3 surface in characteristic $p$ if and only if there exists a primitive embedding $S\hookrightarrow N_{p,1}$. 

\begin{proposition}
\label{proposition: superspecial list}
A H\"ohn--Mason pair $(G,S)$ admits primitive embedding $S\hookrightarrow N_{p,1}$ if and only if $(G,S,p)$ belongs to Table \ref{table: sublist of HM}.
\end{proposition}
The calculation method is similar as the proof of \cite[Theorem 4.14]{laza2022symplectic}, where Laza and the second author classified symplectic automorphism groups of smooth cubic fourfolds. 
\begin{proof}
Given a H\"ohn--Mason pair $(G,S)$, then $q_S$ is known. By Proposition \ref{prop:lat criterion}, it suffices to check whether there exists an embedding $S\oplus \BD_{p,\sigma}\hookrightarrow \II_{9,41}$ such that the image of $S$ is primitive. Suppose such an embedding exists.

Recall that $q_{\BD_{p,1}}=-q_{N_{p,1}}$ is equal to $p^{+2}$ if $p\equiv 3\;(\md\; 4)$; to $p^{-2}$ if $p\equiv 1\;(\md\; 4)$. Let $\widetilde{S}$ be the saturation of the image of $S\oplus \BD_{p,1}$ in $\II_{9,41}$. Denote by $\widetilde{S}^{\perp}$ the orthogonal complement of $\widetilde{S}$ in $\II_{9,41}$. Then $\sgn(\widetilde{S}^{\perp})=(1, 21-\rank(S))$ and $q_{\widetilde{S}^{\perp}}=-q_{\widetilde{S}}$.

Depending on $q_S$, there are at most three possibilities for $q_{\widetilde{S}}$, corresponding to $[\widetilde{S}: S\oplus \BD_{p,1}]=1,p^2, p^4$. For each possibility, we can use Nikulin's existence theorem \cite[Theorem 1.10.1]{Nik79} to check whether an even lattice of rank $(1, 21-\rank(S))$ and with discriminant form $-q_{\widetilde{S}}$ exists. If not, then the corresponding possibility is excluded. If so, suppose $K$ is an even lattice of rank $(1, 21-\rank(S))$ and $q_K=-q_{\widetilde{S}}$. Then there exists a saturation $\widetilde{S}\oplus K\hookrightarrow \II_{9,41}$. Since $S$ is primitive in $\widetilde{S}$, and $\widetilde{S}$ is primitive in $\II_{9,41}$, we obtain an embedding $S\oplus \BD_{p,1}\hookrightarrow \II_{9,41}$ with image of $S$ primitive.
\end{proof}

\begin{example}
Take $(G, q_S)\cong (L_3(4).2, 4_3^{-1} 3^{-1} 7^{-1})$ for an example (the corresponding number in the main list of H\"ohn--Mason is 170). In this case $p$ is supposed to be $3$ or $7$. Suppose $p=3$, then $\widetilde{S}$ is a saturation of $S\oplus \BD_{3,1}$ of index $3$. We have $q_{S\oplus \BD_{3,1}}=4_3^{-1} 3^{-1} 7^{-1}\oplus 3^{+2}=4_3^{-1} 3^{-3} 7^{-1}$, which implies that $q_{\widetilde{S}}=4_3^{-1} 3^{+1} 7^{-1}$, and $-q_{\widetilde{S}}=4_5^{-1} 3^{-1} 7^{+1}$. If $p=7$, we have $q_{S\oplus \BD_{7,1}}=4_3^{-1} 3^{-1} 7^{-3}$ and $-q_{\widetilde{S}}=4_5^{-1} 3^{+1} 7^{-1}$. Notice that the rank one lattice $\langle 84\rangle$ has discriminant form $4_5^{-1} 3^{+1} 7^{-1}$. Therefore, $(G, q_S)\cong (L_3(4).2, 4_3^{-1} 3^{-1} 7^{-1})$ is realized for $p=7$ and not possible for $p=3$. 
\end{example}

\subsection{Artin Invariants $\sigma\ge  2$}
In this subsection, let us take a glimpse at supersingular K3 surfaces with Artin invariants $\sigma\ge 2$. Let $X$ be such a supersingular K3. Then $\NS(X)\cong N_{p,\sigma}$. 

Suppose $G\le \Aut^s(X)$ is a finite subgroup. Then there is an induced action of $G$ on $N_{p,\sigma}$ with covariant lattice $S_G(N_{p,\sigma})$. Since the action of $G$ on $A_{N_{p,\sigma}}$ is trivial, there exists a saturation $N_{p,\sigma}\hookrightarrow N_{p, 1}$ of index $p^{2\sigma-2}$, such that the action of $G$ on $N_{p,\sigma}$ extends to an action on $N_{p,1}$. 

The covariant lattice $S_G(N_{p,1})$ is a saturation of $S_G(N_{p,\sigma})$. Both $S_G(N_{p,\sigma})$ and $S_G(N_{p,1})$ are negative definite. The lattice $S_G(N_{p,\sigma})$ has no roots, while $S_G(N_{p,1})$ may have. Denote by $R_G$ the root sublattice of $S_{G}(N_{p,1})$. Denote by $R_G^{\perp}$ the orthogonal complement of $R_G$ in $S_G(N_{p,1})$. Clearly, the action of $G$ on $S_G(N_{p,1})$ preserves $R_G$ and $R_G^{\perp}$. Thus we have a group homomorphism $G\to \Aut(R_G)$. Let $G_L$ and $G_R$ be the kernel and image of this homomorphism, i.e.,  we have an exact sequence of groups:
\begin{equation}\label{eq:decomposition}
1\longrightarrow G_L\longrightarrow G\longrightarrow G_R\longrightarrow 1.
\end{equation}
The lattice $R_G^{\perp}$ contains no roots. The covariant sublattice $S_{G_L}(N_{p,1})$ is a primitive sublattice of $R_G^{\perp}$. Thus $(G_L, S_{G_L}(N_{p,1}))$ is a Leech pair. Since $\rank(N_{p,1})+\ell(A_{N_{p,1}})= 24$ and  $S_{G_L}(N_{p,1})$ is a primitive sublattice of $N_{p,1}$, we have $\ell(A_{S_{G_L}(N_{p,1})})+\rank(S_{G_L}(N_{p,1}))\le 24$ by Lemma \ref{lemma: sum}. Then by Proposition \ref{prop:emb leech}, the Leech pair $(G_L, S_{G_L}(N_{p,1}))$ is a subpair of $(\Co_0, \BL)$. Moreover we have the following characterization of $(G_L, S_{G_L}(N_{p,1}))$. 

\begin{proposition}
\label{proposition: G_L}
The saturation $(\widetilde{G}_L, S_{G_L}(N_{p,1}))$ of $(G_L, S_{G_L}(N_{p,1}))$ is a H\"ohn--Mason pair belonging to Table \ref{table: sublist of HM}.
\end{proposition}
\begin{proof}
This is because $S_{G_L}(N_{p,1})$ is a primitive sublattice of $N_{p,1}$.
\end{proof}

\begin{lemma}
\label{lemma: supersingular to root pair}
	The root lattice $R_G$ satisfies the following statements:
	\begin{enumerate}
		\item $(R_G)^{G_R}=\{0\}$,
		\item the sublattice generated by $pR_G$ and elements of the form $g\alpha-\alpha$ for $g\in G,\alpha\in R_G$ contains no roots.
	\end{enumerate}
\end{lemma}
\begin{proof}
	The first statement follows from that $R_G\subset S_{G}(N_{p,1})$.
	
	For the second statement, we show that the sublattice is contained in $S_G(N_{p,\sigma})$. Since $N_{p,\sigma}\supset pN_{p,1}$, we have $S_G(N_{p,\sigma})\supset pS_G(N_{p,1})$. Thus $R_G\cap S_{G}(N_{p,\sigma})\supset pR_G$. 
	
	Since $G$ acts trivially on $A_{S_{G}(N_{p,\sigma})}$, the equivalence class of $g\alpha-\alpha$ in $A_{S_G(N_{p,\sigma})}$ is trivial for all $g\in G$ and $\alpha\in S_G(N_{p,1})$. This implies that $g\alpha-\alpha\in S_G(N_{p,\sigma})$. We conclude that the lattice generated by $pR_G$ and elements of the form $g\alpha-\alpha$ for $g\in G,\alpha\in R_G$ is contained in $S_G(N_{p, \sigma})$. Since there is no root in $S_{G}(N_{p,\sigma})$, we are done.
\end{proof}
\begin{remark}
    The automorphisms of root sytems of type $A,D,E$ will come into play. Later in \S\ref{section: p-root pair}, we will systematically study root systems which admit a finite group action satisfying the conditions in Lemma \ref{lemma: supersingular to root pair}.
\end{remark}
\section{The Tame Case}
\label{section: tame}
In this section we give the proof of Proposition \ref{proposition: tame supersingular} and Proposition \ref{proposition: finite height and tame}. 

\subsection{Proof of Proposition \ref{proposition: tame supersingular}: Supersingular K3 Surfaces}
The next lemma will be useful.

\begin{lemma}
\label{lemma: gcd(G,p)=1}
If $\gcd(|G|, p)=1$, then $\ell_p(A_{S_G(X)})=0$, and $S_G(X)\cong S_G(N_{p,1})$.
\end{lemma}
\begin{proof}
Take any $[x]\in A_{S_G(X)}$ with $x\in S_G(X)_{\BQ}$. The sum $\sum\limits_{g\in G} gx \in S_G(X)_{\BQ}$ is fixed by $G$, hence equals to zero. On the other hand, since $G$ induces a trivial action on $A_{S_G(X)}$, we must have $gx-x\in S_G(X)$ for any $g\in G$. Therefore, $\sum\limits_{g\in G} (gx-x) \in S_G(X)$, which implies that $|G| x\in S_G(X)$ where $\gcd(|G|, p)=1$. We conclude that $\ell_p(A_{S_G(X)})=0$. Since $S_G(X)\subset S_G(N_{p,1})$ is of index a $p$-power, and $\gcd(|A_{S_G(X)}|, p)=1$, we must have $S_G(X)\cong S_G(N_{p,1})$.
\end{proof}

By Lemma \ref{lemma: leech pair from geo}, the pair $(G, S_G(X))$ is a Leech pair. We have $\rank(N_{p,1})+\ell_q(A_{N_{p,1}})= 22$ if $q\ne p$. By Lemma \ref{lemma: sum} we have $\rank(S_G(N_{p,1}))+\ell_q(A_{S_G(N_{p,1})})\le 22$. It is clear that $\rank(S_G(N_{p,1}))+\ell_p(A_{S_G(N_{p,1})})\le 22$ since $\ell_p(A_{S_G(N_{p,1})})=0$. Therefore, $\rank(S_G(N_{p,1}))+\ell(A_{S_G(N_{p,1})})\le 22$. From Lemma \ref{lemma: gcd(G,p)=1} and Proposition \ref{prop:emb leech}, the Leech pair $(G, S_G(X))$ is isomorphic to a subpair of $(\Co_0, \BL)$. Since $S_G(X)$ admits a primitive embedding into $N_{p,1}$, the saturation of $(G, S_G(X))$ is a H\"ohn--Mason pair belonging to Table \ref{table: sublist of HM}.

We have a saturation $S_G(X)\oplus \NS(X)^G\subset \NS(X)$. By Lemma \ref{lemma: gcd(G,p)=1}, $\ell_p(A_{S_G(X)})=0$, hence $\ell_p(\NS(X)^G)=\ell_p(\NS(X))=2\sigma$. This implies that $\rank(\NS(X)^G)\ge 2\sigma$. We conclude that $\rank(S_G(X))\le 22-2\sigma$. Now the proof of Proposition \ref{proposition: tame supersingular} is complete.

\subsection{Proof of Proposition \ref{proposition: finite height and tame}: K3 Surfaces of Finite Height}
We mainly use Jang \cite{jang2017lifting} to turn this into a problem over characteristic 0. For more on K3 surfaces of finite height, see \cite{AM77} or \cite[\S 6]{Lie16} and references therein.  

The following lemma is essentially due to Jang (\cite[Theorem 3.2, Corollary 3.4]{jang2017lifting}):
\begin{lemma}
\label{lemma: finite height}
Suppose $G\le \Aut^s(X)$ is a finite subgroup such that $\gcd(|G|, p)=1$. Then we have the following.
\begin{enumerate}
\item The induced action of $G$ on $\NS(X)$ is faithful.
\item The induced action of $G$ on $A_{\NS(X)}$ is trivial. 
\item $\NS(X)$ admits a primitive embedding into the K3 lattice $U^3\oplus E_8^2$, in particular we have $\rank(\NS(X))+\ell(A_{\NS(X)})\le 22$.
\end{enumerate}
\end{lemma}
\begin{proof}
For any $g\in G$, by the proof of \cite[Corollary 3.4]{jang2017lifting}, there exists a N\'eron--Severi group preserving lifting $\mathscr{X}/W(k)$ of $X$, i.e. $\NS(\mathscr{X}_{\overline{K}})\cong \NS(X)$ and the action of $g$ on $X$ can be lifted to an action on $\mathscr{X}/W(k)$ compatible with the specialization map $\NS(\mathscr{X}_{\overline{K}})\cong \NS(X)$.
	
By Lieblich--Maulik \cite[Theorem 2.1, Corollary 2.4]{LM18} (see also Jang \cite[Page 1790]{jang2017lifting}), we have an inclusion $\Aut(\mathscr{X}_{\overline{K}})\hookrightarrow\Aut(X)$. In particular, $\Aut(\mathscr{X})\hookrightarrow\Aut(X)$ where $\Aut(\mathscr{X})$ is the automorphism group of $\mathscr{X}$ as a $W(k)$-scheme.

Since $g$ is of finite order, then we know that its lifting, still denoted by $g$, is also of finite order (coprime with $p$). We denote the natural representation $\rho:\Aut(\mathscr{X})\rightarrow \GL(H^{0}(\mathscr{X},\omega_{\mathscr{X}}))=W(k)^{\times}$, then  $\rho(g) $ acts as a root of unit in $W(k)^{\times}$ since $g$ of finite order. But $g$ is symplectic on the special fiber $X$, hence the eigenvalue is 1 in the residue field $k$. Then it has to be 1 in $W(k)$ since the order of $g$ is coprime with $p$. Therefore, $\rho(g)=Id$. Thus $g$ is also symplectic when restricted to $\mathscr{X}_{\overline{K}}$.
	
Next we show that $g$ acts trivially on $A_{\NS(\mathscr{X}_{\overline{K}})}$. Since such a conclusion is true over $\mathbb{C}$, we can use the "Lefschetz principle" here to conclude over an algebraically closed field of characteristic 0. To be more precise, there exists a finitely generated field $E$ over $\mathbb{Q}$ contained in $\bar{K}$ and a K3 surface $X'$ defined over $E$ with $X'\times _{E}\overline{K}\cong \mathscr{X}_{\overline{K}}$ and $g$ acts on $X'$ symplectically and of finite order. Since $E$ is finitely generated over $\mathbb{Q}$, there is also a field embedding $E\hookrightarrow \mathbb{C}$. We denote by $X'_{\BC}=X'\times_{E}\mathbb{C}$. By \cite[\S 17, Lemma 2.2]{Huy16}, we have $\NS(\mathscr{X}_{\overline{K}})\cong \NS(X'_{\overline{E}})\cong \NS (X'_{\BC})$, and the action of $g$ on $X'_{\BC}$ is symplectic. We then know that the induced action of $g$ on $H^2(X'_{\BC}, \BZ)$ fixes all elements in the transcendental lattice $T(X'_{\BC})$. Therefore, the induced action of $g$ on $A_{\NS(X)}\cong A_{\NS(X'_{\BC})}\cong A_{T(X'_{\BC})}$ is trivial.

 Each $g\in G$ acts faithfully on $\NS(X'_\BC)$, hence also faithfully on $\NS(X)$. From the isomorphism $\NS(X)\cong \NS(X'_\BC)$ we also know that $\NS(X)$ admits a primitive embedding into the K3 lattice.
\end{proof}

Applying Lemma \ref{lemma: sum} for the primitive sublattice $S_G(X)\subset \NS(X)$, we have $$\rank(S_G(X))+\ell(A_{S_G(X)})\le 22.$$ By Lemma \ref{lemma:trivial action}, we know that the induced action of $G$ on $A_{S_G(X)}$ is trivial. Then it follows from Proposition \ref{prop:emb leech} that $(G, S_G(X))$ is a subpair of $(\Co_0, \BL)$. It is clear from Lemma \ref{lemma: finite height} that $S_G(X)$ also admits a primitive embedding into the K3 lattice.

Notice that $\rank(\NS(X))\le 22-2h$ (see \cite[Page 130]{AM77}). Since $G$ is finite, we have $\rk (S_G(X))\le 21-2h$. 

Finally, since $\gcd(|G|,p)=1$, we know $\gcd(|A_{S_G(X)}|,p)=1$, similar as Lemma \ref{lemma: gcd(G,p)=1}. We hence finish the proof of Proposition \ref{proposition: finite height and tame} for K3 surfaces of finite height.

\section{$p$-Root Pairs}
\label{section: p-root pair}

We start this section with the following definition motivated by Lemma \ref{lemma: supersingular to root pair}.
\begin{definition}
Let $H$ be a finite group acting faithfully on a root lattice $R$. We call $(H,R)$ a \emph{pseudo $p$-root pair} if the lattice $R_H^\#$ generated by $pR$ and  $\{g\alpha-\alpha\big{|} g\in H,\alpha\in R\}$ contains no roots. We call $(H,R)$ a \emph{$p$-root pair} if it is a pseudo $p$-root pair, and in addition $R^H=\{0\}$.

We say a (pseudo) $p$-root pair $(H,R)$ is irreducible if $R$ is an irreducible root lattice (namely, the associated Dynkin diagram is connected). We will simply write $R^{\#}_H$ as $R^{\#}$ if no ambiguity caused.
\end{definition}
Suppose $R=\oplus R_i$, then we define $H_i$ to be the image of $\{g\in H\big{|} g(R_i)=R_i\}$ in $\Aut(R_i)$, which is called the stabilizer of $R_i$.
\begin{lemma}
\label{lemma: component inherit the structure}
If $(H,R)$ is a pseudo $p$-root pair, then each $(H_i, R_i)$ is also a pseudo $p$-root pair. If $(H,R)$ is a $p$-root pair, then each $(H_i, R_i)$ is also a $p$-root pair.
    \end{lemma}
    \begin{proof}
The first statement is clear. Suppose $(H,R)$ is a $p$-root pair, to show that $(H_i, R_i)$ is a $p$-root pair, we only need to show $R_i^{H_i}=\{0\}$. Otherwise there is a nonzero vector $v\in R_i^{H_i}$. The factor of $\sum\limits_{g\in H} gv$ in $R_i$ equals to $|\{g\in H\big{|} g(R_i)=R_i\}| v$, hence nonzero. Meanwhile $\sum\limits_{g\in H} gv\in R^H=\{0\}$, a contradiction.
\end{proof}
In the following, we will first discuss irreducible $p$-root pairs and then general $p$-root pairs. In particular, in the last subsection, we prove the last statement in Theorem \ref{theorem: main3}. Actually we here show that it holds more generally for pseudo $p$-root pairs .  
\subsection{$p$-Root Pairs of Type $A$}
We first study (pseudo) $p$-root pairs $(H,R)$ of type $A$. This means that $R$ is a root lattice of type $A_m$ for certain positive integer $m$. Then we can realize $R$ as the subgroup of the standard Euclidean space $\BR^{m+1}$ generated by $v_i-v_j$. Here $(v_1, \cdots, v_{m+1})$ is the standard basis of $\BR^{m+1}$. The Weyl group $W(R)$ of $R=A_m$ can be identified with the permutation group $S_{m+1}$ of vectors $v_1, \cdots, v_{m+1}$. The roots $\alpha_i=v_i-v_{i+1}$ for $1\le i\le m$ form a base of the root lattice.
\begin{lemma}\label{lem: pseudo A}
Let $(H,R)$ be a pseudo $p$-root pair, with $R=A_{m}$. Then $H\cap W(A_m)$ is a $p$-group, and for any nontrivial $g\in H\cap W(A_m)$, it has no fixed element in $\{v_1,\cdots, v_{m+1}\}$. In particular, it implies that if $H\cap W(A_m)$ is nontrivial, then $p|(m+1)$.   	
    \end{lemma}
    \begin{proof}
We first prove that $H\cap W(A_m)$ is a $p$-group. Suppose there is an element $g\in H\cap W(A_m)$ of order $q$ where $q$ is a prime and $q\neq p$. Without loss of generality, we assume that $g$ has a cycle $(v_{1},v_{2},\cdots,v_{q})$. Then 
\begin{equation*}
\sum _{i=1}^{q}(g^{i}(\alpha_{1})-\alpha_{1})=-q\alpha_{1}\in R^\#.
\end{equation*}
Since $p\alpha_{1}\in R^\#, p\neq q$, we deduce $\alpha_{1}\in R^\#$, a contradiction. So $H\cap W(A_m)$ is a $p$-group. Therefore, $\abs{H}=p^{n}$ or $\abs{H}=2\cdot p^{n}$ for certain integer $n\ge 0$.
    	
Next, we prove that any nontrivial $g\in H\cap W(A_m)$ has no fixed element in $\{v_1,\cdots, v_{m+1}\}$. Otherwise we may choose an invariant element which we may denote as $v_{m+1}$. Since $g$ is nontrivial, there exists $v_i$, such that $g(v_i-v_{m+1})-(v_i-v_{m+1})$ is a root, a contradiction. The last statement hence follows.
    \end{proof}
\begin{corollary}
    Let $(H,R)$ be a pseudo $p$-root pair, with $R=A_{m}$. For any $g\in H\cap W(A_m)$, if $g$ is nontrivial, then all the orbits of $g$ have same length.
\end{corollary}
\begin{proof}
    In the proof of the above lemma, we have shown that for any element $g\in H\cap W(A_m)$, $g$ has no fixed element in $\{v_1,\cdots, v_{m+1}\}$. Apply this result to all the powers of $g$, then we can conclude.
\end{proof}
\begin{remark}
    This is a quite restrictive condition on $H$. However, $H$ needs not to be cyclic. For example, \[
    H=\{id, (12)(34), (13)(24), (14)(23)\}\subset S_4.
    \]
    Cycles of any nontrivial element have the same length, but $H$ is not cyclic.
\end{remark}
For irreducible $p$-root pairs of type A, we have a stronger conclusion:
    \begin{proposition}\label{prop: type A}
Given a $p$-root pair $(H,R)$ with $R=A_{m}$, then $(m+1)\big{|} |H\cap W(A_m)|$. In particular, $m+1$ is a power of $p$.
\end{proposition}
\begin{proof}
If $H\le W(A_m)$, we claim that $H$ acts transitively on $\{v_1, \cdots, v_{m+1}\}$. Otherwise, there exist two orbits $U,V\subset \{v_1, \cdots, v_{m+1}\}$ for the action of $H$. Take $v_i\in U, v_j\in V$. Consider the root $v_i-v_j$. Since $R^H=\{0\}$, we have $\sum\limits_{g\in H} g(v_i-v_j)=0$. Then 
\begin{equation*}
\sum\limits_{g\in H} gv_i=\sum\limits_{g\in H} gv_j,
\end{equation*}
But this is impossible since the left side is a positive combination of elements in $U$, while the left is a positive combination of elements in $V$. Therefore, $G$ acts transitively on $\{1,2,\cdots,m+1\}$ and hence $(m+1)\big{|} |H\cap W(A_m)|$.

If $H\not\subset W(A_m)$. Let $H'=H\cap W(A_{m})$. Then $|H|=2|H'|$ is even. So there exists an element of order $2$ in $H$, which is not in $W(A_m)$. So we can choose an involution $w\in W(A_m)$ such that $-w\in H$. Then $H=H'\cup (-wH')$. 

Since $R^{H}=\{0\}$, then for $x\in R$, we have $\sum\limits_{g\in H}g\cdot x=0$. We claim that $H'$ acts transitively on $\{1,\cdots,m+1\}$. Otherwise, we consider two orbits of $H'$, say $U,V$ and take $v_i\in U, v_j\in V$. Let $x=v_{i}-v_{j}$, then we have:	
	\begin{align*}
		\sum_{h\in H'} h\cdot v_{i}-\sum_{h\in H'}wh\cdot v_{i}-\sum_{h\in H'}h\cdot v_{j}+\sum_{h\in H'}wh\cdot v_{j}= 0
	\end{align*}
    Since $H'$ is normal in $H$, it is straightforward to see that $w$ sends an $H'$-orbit to an $H'$-orbit. By the equality above, we know that $w(U)=U, w(V)= V$. But $H'$ is a $p$-group, thus both $\abs{U}, \abs{V}$ are odd, and since $w$ is of order 2, there exists $v_{i'}\in U, v_{j'}\in V$ such that $w(v_{i'})=v_{i'}, w(v_{j'})=v_{j'}$. Then 
    $$-w(v_{i'}-v_{j'})-(v_{i'}-v_{j'})=-2(v_{i'}-v_{j'})\in R^\#.$$ Since $p(v_{i'}-v_{j'})\in R^\#$, we have $v_{i'}-v_{j'}\in R^\#$, a contradiction. Thus we know $H'$ acts transitively on $\{1,2,\cdots,m+1\}$, which implies $(m+1)\mid \abs{H\cap W(A_{m})}$. 
    \end{proof}
\begin{proposition}
    If $(H,R)$ is a $p$-root pair with $R=A_m$, then $(H\cap W(A_m), R)$ is also a $p$-root pair.
\end{proposition}
\begin{proof}
   We put $H'=H\cap W(A_m)$ and we only need to show that $R^{H'}=\{0\}$. From the proof of Proposition \ref{prop: type A}, we know that $H'$ acts transitively on $\{1,2,\ldots, m+1\}$. For $x\in R^{H'}$, we write $x=\sum\limits_{i=1}^{m+1} x_iv_i$. By the transitivity of $H'$-action, and $x\in R^{H'}$, we know that $x_i$ are the same for $1\le i\le m+1$, hence they are all 0. 
\end{proof}
\begin{remark}
    This may not hold for other root systems. See Lemma \ref{lem:p=3,D4}.
\end{remark}
\subsection{$p$-Root Pairs of Type $D_m$ with $m\ge 5$}
In this subsection we study (pseudo) $p$-root pairs of type $D$. Let $v_1, \cdots, v_m$ be the standard basis for the standard Euclidean space $\BR^m$. Recall that the root lattice $D_m$ ($m\ge 4$) can be realized as the subgroup of $\BR^m$ generated by $\pm v_i\pm v_j$, $1\le i<j\le m$. We choose the set of simple roots as $\{\alpha_1,\cdots,\alpha_m\}$ with $\alpha_i=v_i-v_{i+1}, 1\le i\le m-1, \alpha_m=v_{m-1}+v_m$ .

For $g\in S_m$ acting on $\{v_1, \cdots, v_m\}$, and $\epsilon_1, \cdots, \epsilon_m\in\{\pm 1\}$, there is an element in $\GL(\BR, m)$ given by $v_i\mapsto \epsilon_i g(v_i)$. These elements form a finite subgroup $S_m\ltimes (\BZ/2\BZ)^m$ of $\Aut(D_m)$. When $m\ge 5$, we have an equality $\Aut(D_m)=S_m\ltimes (\BZ/2\BZ)^m$. When $m=4$, we have $[\Aut(D_4):S_4\ltimes (\BZ/2\BZ)^4]=3$.

The Weyl group $W(D_m)$ of $D_m$ is isomorphic to $S_m \ltimes (\BZ/2\BZ)^{m-1}$. For any element $g\in S_m\ltimes (\BZ/2\BZ)^m$, we denote by $\overline{g}$ its part in $S_m$.

\begin{lemma}
\label{lemma: type D}
Suppose $(H, R)$ is a pseudo $p$-root pair of type $D_m$ for $m\ge 4$. Then for any $g\in H\cap (S_m\ltimes (\BZ/2\BZ)^m)$, we have $\overline{g}=id$. And there exists at most one $i\in\{1,\cdots,m\}$ such that $g(v_i)=-v_i$.
\end{lemma}
\begin{proof}
If there is an element $g\in H\cap (S_m\ltimes (\BZ/2\BZ)^m)$ whose image $\overline{g}$ in $S_{m}$ has an orbit of length at least 3, then without loss of generality, we assume that $\overline{g}(v_1)=v_2$ and $\overline{g}(v_2)=v_3$.

If $g(v_1)=v_2$, then $g(v_1+v_2)-(v_1+v_2)$ is a root. If $g(v_1)=-v_2$, then $g(v_1-v_2)-(v_1-v_2)$ is a root. Hence we rule out this possibility.

If $g\in H\cap (S_m\ltimes (\BZ/2\BZ)^m)$ such that $\overline{g}$ has a cycle of length two. We assume $\overline{g}(v_{1})=v_{2}$ and $\overline{g}(v_{2})=v_{1}$.

If $g(v_1)=v_2$ and $g(v_{2})=v_{1}$, we have $g(v_{1}-v_{2})-(v_{1}-v_{2})=-2(v_{1}-v_{2})$. Since $p(v_{1}-v_{2})\in R^\#$, we deduce $v_{1}-v_{2}\in R^\#$, a contradiction. 

If $g(v_1)=v_2$ and $g(v_{2})=-v_{1}$, then $g(v_{1}-v_{2})-(v_{1}-v_{2})=2v_{2}\in R^\#$, and $g(v_{1}+v_{2})-(v_{1}+v_{2})=-2v_{1}\in R^\#$, thus $2(v_{1}-v_{2})\in R^\#$, still a contradiction. The remaining cases are similar. We conclude that $\overline{g}=id$.

Suppose $g$ has more than one sign changes. Suppose $g(v_{1})=-v_{1}, g(v_{2})=-v_{2}$. Then $g(v_{1}-v_{2})-(v_1-v_2)=-2(v_{1}-v_{2})\in R^\#$. Since $p(v_{1}-v_{2})\in R^\#$, we have $v_{1}-v_{2}\in R^\#$, a contradiction. The lemma follows.
\end{proof}
    
 \begin{proposition}\label{prop: type D}
Let $(H,R)$ be a pseudo $p$-root pair with $R=D_m$ for $m\ge 4$, then $H\cap W(R)=\{id\}$.
    \end{proposition}
    \begin{proof}
We have $W(D_m)=S_m\ltimes (\BZ/2\BZ)^{m-1} < S_m\ltimes (\BZ/2\BZ)^m$ consisting of elements with even number of sign changes. Suppose $g\in H\cap W(A_m)$, then by Lemma \ref{lemma: type D}, we have $\overline{g}=id$ and $g$ has at most one sign change. The only possibility is $g=id$. The proposition follows.
    \end{proof}

\begin{proposition}
\label{proposition: no p-root pair of type D}
Suppose $R=D_m, m\ge 5$, then $(H,R)$ is a pseudo $p$-root pair if and only if:
\begin{enumerate}
    \item $H=\{id\}$ or,
    \item $H=\langle g\rangle$ with $g(v_i)=-v_i$ for a unique $i$, and $g(v_j)=v_j$ for $j\in\{1,\cdots,m\}-\{i\}$. 
\end{enumerate}

Moreover, there is no $p$-root pair of type $D_m$ with $m\ge 5$.
\end{proposition}
\begin{proof}
If $m\ge 5$, we have $\Aut(D_m)=S_m \ltimes (\BZ/2\BZ)^m$. By Lemma \ref{lemma: type D}, any element $g\in H$ belongs to the two possibilities. It is straightforward to verify that the listed $H$ indeed form pseudo $p$-root pairs.

The last statement follows if we add one more restriction that $R^{H}=\{0\}$.
    \end{proof}

\begin{proposition}\label{prop:no D}
    Any $p$-root pair $(H,R)$ does not have a component of type $D_m$ with $m\ge 5$.
\end{proposition}
\begin{proof}
This directly follows from Lemma \ref{lemma: component inherit the structure} and Proposition \ref{proposition: no p-root pair of type D}.
\end{proof}
\subsection{ $p$-Root Pair of Type $D_4$}\label{sec:D4}
We have that $\Aut(D_4)=S_3\ltimes W(D_4)$. Consider the outer automorphisms: $x: \alpha_{3}\mapsto\alpha_{4}\mapsto\alpha_{1}\mapsto\alpha_{3}, \alpha_{2}\mapsto\alpha_{2}$,  $y: v_4\rightarrow -v_4$, i.e., $y:\alpha_3\mapsto \alpha_4\mapsto\alpha_3$. Then the outer automorphism group $S_3$ is generated by $x,y$. 

We take an element $ g\in\Aut(D_4)$ such that $g\alpha_2=-2\alpha_2-\alpha_1-\alpha_3-\alpha_4, g(\alpha_i)=\alpha_i+\alpha_2, i=1,3,4$. Spell this out, we have $g:-v_1\mapsto v_3\mapsto v_2$. In other words $g=(v_1\mapsto -v_1)\circ (132)\circ (v_1\mapsto -v_1)$. 
\begin{lemma}
    $\langle g\rangle\times \langle x\rangle$ is a Sylow 3-subgroup of $\Aut(D_4)$. And any subgroup of $\Aut(D_4)$ of order 3 is conjugate to one of the follows: $\langle g\rangle,\langle x\rangle, \langle gx\rangle$.
\end{lemma}
\begin{proof}
    We first verify that $gx=xg, g^{3}=id$, and $\langle g, x\rangle\cong \BZ/3\BZ\times\BZ/3\BZ$ is a Sylow 3-subgroup of  $\Aut(D_4)$. Any subgroup of $\Aut(D_4)$ of order 3 has to be conjugate to one of $\langle g\rangle,\langle x\rangle, \langle gx\rangle, \langle gx^2\rangle$. 
    
    Notice that $g$ acts identically on $\langle v_4, v_1-v_2-v_3\rangle $ which contains no roots, while $x$ acts identically on $\{\alpha_2,\alpha_1+\alpha_3+\alpha_4\}$. We can calculate that $gx, gx^2$ do not fix any elements in $R$. Hence $\langle x \rangle, \langle g\rangle, \langle gx\rangle$ are not conjugate to each other in $\Aut(D_4)$. Recall that $y\in \Aut(D_4): v_4\mapsto -v_4$, we can verify that $y(gx)y^{-1}=gx^{2}$.    
    \end{proof}
\begin{proposition}\label{prop: p>=5 no d4}
    If $p\ge 5$, there is no $p$-root pair of type $D_4$.
\end{proposition}
\begin{proof}
By Proposition \ref{prop: type D}, we know that for a pseudo $p$-root pair $(H,D_4)$, $H$ is isomorphic to a subgroup of $\langle x,y\rangle=S_3$ induced by the following exact sequence:
\[
1\rightarrow W(D_4)\rightarrow \Aut(D_4)\rightarrow S_3\rightarrow 1
\]

We can verify that $R^\#$ has root $\alpha_2$ for $(\langle g\rangle, D_4)$. Hence $(\langle g\rangle, D_4)$ is not a pseudo-root pair. For $(\langle gx\rangle, D_4)$, 
\[
R^\#=\langle  3D_4, 2\alpha_1+\alpha_2+\alpha_3,\alpha_1+\alpha_3+\alpha_4, pD_4\rangle
 \]
Hence if $p\ne 3$, $R^\#$ has roots for $(\langle gx\rangle, D_4)$. 

We first assume that $|H|=3$, then by the previous arguments, since $p\ge 5$, $H$ has to be conjugate to $\langle x\rangle$ and then $R^{H}\ne 0$. Now suppose that $|H|=6, \langle x\rangle\subset H$ is a normal subgroup, and let $\gamma\in H$ be an order 2 element. Since $\langle x\rangle$ is normal, then $\gamma$ acts on $R^{x}$ which contains $\BZ\alpha_2+\BZ(\alpha_1+\alpha_3+\alpha_4)$. Since $\alpha_2\in R^{x}$, $\gamma$ cannot acts as $-id$ on $R^{x}$. Hence $R^{\langle \gamma,x\rangle}\ne 0$.
Suppose $|H|=2$. Since $R^{H}=\{0\}$, hence it acts as $-id$ on $R$. Then $R^{\#}$ contains roots.  To conclude, there is no $p$-root pair of type $D_4$ if $p\ge 5$.
\end{proof}
However, we have
    \begin{lemma}\label{lem:p=3,D4}
        If $p=3$, $(\langle gx\rangle, D_4)$ is a $p$-root pair.
    \end{lemma}
  \begin{proof}
    For $(\langle gx\rangle, D_4)$
    \begin{align*}
    R^\# &=\langle 3D_4, 2\alpha_1+\alpha_2+\alpha_3,\alpha_1+\alpha_3+\alpha_4\rangle\\
    &=\langle 3D_4, 2v_1-v_2-v_4, v_1-v_2+2v_3\rangle
    \end{align*}
    hence for all elements $\lambda\in R^\#$, suppose that $\lambda=a(2v_1-v_2-v_4)+b(v_1-v_2+2v_3)+\cdots$ we have that:
    \[
    \lambda\equiv (2a+b)v_1+(-a-b)v_2+2bv_3+(-a)v_4 \;(\md\; 3D_4).
    \]
    If $\lambda$ is a root, then exactly two of $\{(2a+b), (-a-b), 2b, (-a)\}$ has to be $0$ in $\BZ/3\BZ$ which is impossible.  
  \end{proof}
   
    \subsection{$p$-Root Pairs of Type $E$.}
For lattices of type $E$, we have $\Aut(E_6)=W(E_6)\rtimes (\BZ/2\BZ)$, $\Aut(E_7)=W(E_7)$, $\Aut(E_8)=W(E_8)$.

\begin{lemma}
\label{lemma: a case S_G has no roots}
Let $(H,R)$ be a pseudo $p$-root pair, such that $|H|$ is coprime to $p$, then $S_H(R)$ has no roots.
\end{lemma}
\begin{proof}
Take any $v\in S_H(R)$. Then $\sum\limits_{g\in H} gv=0$. Combining with $\sum\limits_{g\in H} (gv-v)\in R^\#$, we deduce $|H| v\in R^\#$. Since $\gcd(|H|,p)=1$ and $pv\in R^\#$, we have $v\in R^\#$. Therefore, $v$ is not a root.
\end{proof}

\begin{corollary}
\label{corollary: a way from root pair to Leech pair}
Suppose $(H,R)$ is a pseudo $p$-root pair, $\gcd(|H|,p)=1$ and $H\le W(R)$. Then $(H, S_H(R))$ is a Leech pair.
\end{corollary}
\begin{proof}
Since the induced action of $W(R)$ on the discriminant group $A_R$ is trivial, by Lemma \ref{lemma:trivial action}, we know that the induced action of $H$ on the discriminant group of $S_H(R)$ is trivial. Then the corollary follows from Lemma \ref{lemma: a case S_G has no roots}.
\end{proof}
   
\begin{proposition}
\label{proposition: p group for type E}
Let $(H,R)$ be a pseudo $p$-root pair and $R$ is of type $E_n$ where $n=6,7,8$, then $H\cap W(R)$ is a $p$-group. 
\end{proposition}
\begin{proof}
Suppose not, we choose a $g_0\in H\cap W(R)$, of order $q$ an odd prime different from $p$. We put $H_0=\langle g_0\rangle$. By Corollary \ref{corollary: a way from root pair to Leech pair}, $(H_0,S_{H_0}(R))$ is a Leech pair. Apparently $\rank(S_{H_0}(R))+\ell(A_{S_{H_0}(R)})\le 24$, so $(H_0,S_{H_0}(R))$ is a subpair of  $(\Co_0, \BL)$. However, from H\"ohn-Mason list, this is impossible because $\rank(S_{H_0}(R))< \rank(R)\le 8$.
\end{proof}

    \begin{lemma}
\label{lemma: E}
Suppose $(H,R)$ is a $p$-root pair with $R$ of type $E_n$ ($n=6,7$ or $8$), then either $p=3$ and $R=E_6$ or $E_8$, or $p=5$ and $R=E_8$.
\end{lemma}
\begin{proof}
We have $\abs{W(E_{8})}=696729600=2^{14}\cdot 3^{5}\cdot 5^{2}\cdot 7$, $\abs{W(E_{7})}=2903040=2^{10}\cdot3^{4}\cdot5\cdot7$, $\abs{W(E_{6})}=51840=2^{7}\cdot 3^{4}\cdot 5$. Suppose $(H,R)$ is a $p$-root pair. Then for $R=E_7,E_8$ only $p=3,5,7$ are possible. 

For $R=E_6$, we first show that $H\cap W(E_6)\ne \{id\}$. Suppose that $H\cap W(E_6)= \{id\}$, then $|H|=2$, which implies $S_H(E_6)$ is rootless by Lemma \ref{lemma: a case S_G has no roots}. So $R^H$ is nonzero, contradiction.
By Proposition \ref{proposition: p group for type E}, $|H\cap W(E_6)|$ is a power of $p$. Hence $p\big{|} |W(E_6)|$, $p=3$ or $5$. 

If $p=7$, then $R=E_8$ or $E_7$.  The Sylow $7$-subgroup of $W(R)$ is of order $7$, hence $H\cong \BZ/7\BZ$. Since there are only two irreducible representation of $\BZ/7\BZ$ over $\BQ$, which are of dimension $1$ and $6$, there must exist nonzero invariant vector in $R\otimes\BQ$, contradiction.

If $p=5$, we exclude the possibilities $R=E_6, E_7$ as follows. Suppose $(H,R)$ is $p$-root pair with $R=E_6$ or $E_7$. If $H\le W(R)$, then $H$ is isomorphic to $\BZ/5\BZ$. Consider representations of $H$ over $\BQ$, we know that $R^H$ is nonzero and hence $(H,R)$ cannot be a $p$-root pair. 

If $H$ is not a subgroup of $W(R)$, then we must have $R=E_6$, and the order of $H$ has to be $10$. We put $H'=H\cap W(E_6)$, then $H'$ is a normal subgroup of $H$ of order $5$. The representation $(E_6)_\BQ$ of $H'$ must be a direct sum of a $2$-dimensional invariant subspace and a $4$-dimensional irreducible representation. Thus $\rank(E_6^{H'})=2$. Next we show that $E_6^{H'}$ is not rootless. Since by Sylow theory, every subgroup of $W(E_6)$ with order $5$ is conjugate to $H'$, we only need to find one automorphism of $E_6$ with order $5$ and fixing one root. Since the Dynkin graph of $A_4\oplus A_1$ is a subgraph of the Dynkin graph of $E_6$, there exists a primitive embedding $A_4\oplus A_1 \hookrightarrow E_6$. By Lemma \ref{lemma: extending group action}, any faithful action of $\BZ/5\BZ$ on $A_4$ extends to an action on $E_6$ which fixes $A_1$. Therefore, there exists an automorphism of $E_6$ with order $5$ and fixing at least one root. 

We take an involution $g\in H$. Then $g H' g^{-1}=H'$, and $g(E_6^{H'})=E_6^{H'}$. Since $E_6^{H}=\{0\}$, the involution $g$ is a scalar by $-1$ on $E_6^{H'}$. Take a root $v\in E_6^{H'}$. Then $2v=-g(v)+v\in R^\#$, which implies $v\in R^\#$, contradiction. We conclude that for $5$-root pair $(H,R)$ of $E$-type, $R$ must be $E_8$.

If $p=3$, we show that $R\ne E_7$. Actually, since the exponents of $3$ in $W(E_6)$ and $W(E_7)$ are the same, any action of $3$-group on $E_7$ comes from an action on $E_6\subset E_7$. This implies the existence of a fixed vector in $E_7$, which is a contradiction.
\end{proof}
We can conclude that:
\begin{proposition}
\label{proposition: p>=7 no type E}
Let $(H,R)$ be a $p$-root pair. If $p\ge 7$, then $R$ has no component of type $E$. 
\end{proposition}

We now analyze $p=5$ and $R= E_{8}$. Since $\abs{W(E_{8})}=696729600=2^{14}\cdot 3^{5}\cdot 5^{2}\cdot 7$, the Sylow-$5$ subgroup of $W(E_{8})$ is of order $25$. The following description of $E_{8}$ lattice is referred to Kac \cite[\S 6.7]{Kac90}:
	\begin{enumerate}
		\item $E_{8}=\{\sum k_{i}v_{i}\in \BR^{8}|k_{i}\in\BZ\;\text{or all}\; k_{i}\in 1/2\BZ, \text{and} \;\sum k_{i}\in 2\BZ\}$, 
		\item roots:$\{\pm v_{i}\pm v_{j}, 1\le i<j\le 8\}\cup \{\frac{1}{2}(\pm v_{1}\pm v_{2}\pm\cdots\pm v_{8})|\text{even number of minuses}\}$,
		\item positive simple roots: $\{\alpha_{i}=v_{i+1}-v_{i+2}|i=1,\cdots,6\}$ and $\alpha_{7}=\frac{1}{2}(v_{1}-v_{2}-\cdots-v_{7}+v_{8})$ and $\alpha_{8}=v_{7}+v_{8}$,
		\item the highest root is $\theta=v_{1}+v_{2}$.
	\end{enumerate}
The roots $\alpha_i$ ($1\le i\le 8$) and $-\theta$ form the affine Dynkin graph:
\begin{equation*}
\widetilde{E}_8\colon \dynkin [extended]E8.
\end{equation*}
By deleting the node corresponding to $\alpha_4$, we obtain a subgraph of $\widetilde{E}_8$ which is the Dynkin graph of $A_4\oplus A_4$. So we have an embedding of $A_{4}\oplus A_{4}$ into $E_{8}$. Explictly, $\{-\theta,\alpha_{1},\alpha_{2}, \alpha_{3}\}$ and $\{\alpha_{8},\alpha_{5},\alpha_{6},\alpha_{7}\}$ are naturally bases for the two pieces of $A_4$ respectively.

For any root $\alpha$, we simply put $s_{\alpha}$ as the reflection defined by $\alpha$, i.e., $s_{\alpha}(x)=x-\langle x,\alpha\rangle\alpha$.
\begin{lemma}
	There is a faithful action of $\BZ/5\BZ\times \BZ/5\BZ\subset W(A_{4})\times W(A_{4})$ on $E_{8}$.
\end{lemma} 
\begin{proof}
	We denote the natural generators of $\BZ/5\BZ\times\BZ/5\BZ\subset W(A_4)\times W(A_4)=S_5\times S_5$, by $a$ and $b$. We want to extend the action of $a,b$ on the lattice $A_{4}$ to an action on $E_{8}$.
	
	We define $a(-\theta)=\alpha_{1},a(\alpha_{1})=\alpha_{2},a(\alpha_{2})=\alpha_{3}, a(\alpha_{3})=\theta-\alpha_{1}-\alpha_{2}-\alpha_{3}$, and $a(\alpha_{i})=\alpha_{i}$ for $i=5,6,7,8$. Since $\theta=2\alpha_{1}+3\alpha_{2}+4\alpha_{3}+5\alpha_{4}+6\alpha_{5}+4\alpha_{6}+2\alpha_{7}+3\alpha_{8}$ and $a(-\theta)=\alpha_{1}$, thus $a(\alpha_{4})=-\theta+\sum\limits_{i=1}^{4}\alpha_{i}$. We can verify that $a=s_{-\theta}s_{\alpha_1}s_{\alpha_2}s_{\alpha_3}$, hence it defines an element of $W(E_{8})$, which we still denote by $a$.
	
We define $b(\alpha_i)=\alpha, i=1,2,3, b(\theta)=\theta, b(\alpha_{8})=\alpha_{5}, b(\alpha_{5})=\alpha_{6},b(\alpha_{6})=\alpha_{7},b(\alpha_{7})=-\sum\limits_{i=5}^{8}\alpha_{i}$, thus $b(\alpha_4)=\alpha_4+\alpha_5+\alpha_8$. We can verify that $b=s_{\alpha_8}s_{\alpha_5}s_{\alpha_6}s_{\alpha_7}$, hence it defines an element of $W(E_{8})$, which we still denote by $b$. 

    We can verify that $ab=ba$. For example $ab(\alpha_4)=ba(\alpha_4)=-\theta+\sum\limits_{i=1}^{5}\alpha_i+\alpha_8$. In particular, $\langle a,b\rangle$ is a Sylow-$5$ subgroup of $W(E_8)$.
\end{proof}
\begin{proposition}\label{prop: E contributes 1}
	If $p=5$, $(H,E_{8})$ is a $p$-root pair, then $H\cong \mathbb{Z}/5\mathbb{Z}$. Moreover, there exists a $p$-root pair $(H,E_{8})$ with $H\cong \mathbb{Z}/5\mathbb{Z}$.
\end{proposition}

\begin{proof}
	First notice that any subgroup of $W(E_{8})$ of order $25$ is a Sylow subgroup which is conjugate to the action we construct above. Since $a(\alpha_{4})-\alpha_{4}=-\theta+\alpha_{1}+\alpha_{2}+\alpha_{3}=-v_1-v_5$ is a root, $H$ cannot be a $5$-Sylow subgroup of $W(E_8)$. The last statement follows from \cite[Remark 3.3]{GL11}.
\end{proof}
\subsection{On Reducible $p$-Root Pairs}

\begin{proposition}\label{prop: intersection p group}
   If $(H,R)$ is a pseudo $p$-root pair, then $H\cap W(R)$ is a $p$-group. 
\end{proposition}
\begin{proof}
    Let $H_i$ be as before. Then we have a natural map:
    \[
    H\cap W(R)\rightarrow \prod H_i\cap W(R_i)
    \]
    since $H\cap W(R)$ stabilizes each irreducible component of $R$. It is also injective because if an element stabilizes and acts trivially on each irreducible component  then it acts trivially on $R$. Then the lemma follows from Lemma \ref{lem: pseudo A}, Proposition \ref{prop: type D} and Proposition \ref{proposition: p group for type E}.
\end{proof}
\section{Application of $p$-Root Pairs}
\label{section: application of p-root pair}
Let $G$ be a finite group acting faithfully and symplectically on a supersingular K3 surface $X$ with Artin invariant $\sigma\ge 2$. Recall that the action of $G$ on $N_{p,\sigma}$ can be extended to $N_{p,1}$ since $G$ acts trivially on $A_{N_{p,\sigma}}$. This gives rise to $(G_R, R)$ where $R\subset S_{G}(N_{p,1})$ is the root lattice.  Moreover we have the following exact sequence:
\[
1\rightarrow G_L\rightarrow G\rightarrow G_R\rightarrow 1.
\]
And by Lemma \ref{lemma: supersingular to root pair}, $(G_R,R)$ is a $p$-root pair. In this section, we apply our results on $p$-root pairs to show that when $p\ge 13$, we always have $(p,|G|)=1$. This provides an alternative proof of \cite[Theorem 2.1]{DK09} for supersingular K3. And then we calculate the upper bounds of the $p$-part of $G$ for $p=3, 5, 7, 11$.

\subsection{New Proof for Dolgachev--Keum's Result}\label{sec:new DK}
We start by considering a $p$-root pair $(H,R)$ with $R=A_{p-1}$.
\begin{definition}
For $R=A_{p-1}$ with chosen simple roots $\alpha_i=v_i-v_{i+1}$, we define $T(R)=\{\sum\limits_{i=1}^{p-1}k_i\alpha_i \big{|}k_i\in\BZ, \sum\limits_{i=1}^{p-1}k_i\equiv 0 \;(\md\; p)\}$.
\end{definition}

\begin{lemma}
\label{lemma: disc form of T(A)}
We have 
\begin{equation*}
T(A_{p-1})\subset A_{p-1} \subset A_{p-1}^\vee \subset T(A_{p-1})^\vee
\end{equation*}
with each adjacent inclusion of index $p$. And 
\begin{enumerate}
\item $q_{A_{p-1}}=p^{\epsilon 1}$ with $\epsilon=+$ when $p\equiv 1 \;(\md\; 4)$ and $\epsilon=-$ otherwise. 
\item $q_{T(A_{p-1})}=p^{+3}$.
\end{enumerate}
\end{lemma}
\begin{proof}
Take $\omega_i\in A_{p-1}^\vee$ with $\omega_i(\alpha_j)=\delta_{ij}$, $1\le i,j\le p-1$. We can also regard $\omega_i$ as elements in $(A_{p-1})_\BQ$. Then $A_{p-1}^\vee=\Span_{\BZ}\{\omega_1,\cdots, \omega_{p-1}\}$. We have 
\begin{equation*}
\omega_1=v_1-{1\over p}(v_1+\cdots+v_p)={1\over p}(\alpha_{p-1}+2\alpha_{p-2}+
\cdots+(p-1)\alpha_1). 
\end{equation*}
Then the discriminant group of $A_{p-1}$ is a group of order $p$ with a generator $[\omega_1]$. The value of the finite quadratic form $q_{A_{p-1}}$ at $[\omega_1]$ is $1-{1\over p}\in \BQ/2\BZ$. This shows (1). 

For (2) it suffices to show $p T(A_{p-1})^\vee\subset T(A_{p-1})$. We have 
\begin{equation*}
T(A_{p-1})^\vee=\Span_{\BZ}(\omega_1, \cdots, \omega_{p-1}, {1\over p}\sum\limits_{i=1}^{p-1}\omega_i).
\end{equation*}
Since $p\omega_1=-(\alpha_{p-1}+2\alpha_{p-2}+
\cdots+(p-1)\alpha_1)\in T(A_{p-1})$, we have $pA_{p-1}^\vee\subset T(A_{p-1})$. We can directly check that
\begin{equation*}
\sum\limits_{i=1}^{p-1}\omega_i=\sum\limits_{i=1}^{p-1} i(p-i) \alpha_i \in T(A_{p-1}).
\end{equation*}
The Lemma then follows.
\end{proof}

\begin{lemma}
\label{lemma: action on discriminant of T(A) is nontrivial}
Suppose $(H,R)$ is a $p$-root pair with $R=A_{p-1}$. Then $H\cap W(R)$ has order $p$. Take a generator $g\in H\cap W(R)$, then it is a cycle in $S_p$ of length $p$. Suppose simple roots $\alpha_i=v_i-v_{i+1}$ satisfy $g(v_i)=v_{i+1}$ for $1\le i\le p-1$ and $g(v_p)=v_1$. Then $R^\#=T(A_{p-1})$. And the induced action of $H$ on the discriminant group of $T(A_{p-1})$ is nontrivial.
\end{lemma}
\begin{proof}
The sublattice $T(A_{p-1})$ is generated by $\alpha_1-\alpha_2,\cdots, \alpha_{p-2}-\alpha_{p-1}, pA_{p-1}$. Since
\[
g(\alpha_i)-\alpha_i=\alpha_{i+1}-\alpha_{i}, 1\le i\le p-2, g(\alpha_{p-1})-\alpha_{p-1}=-\sum_{j=1}^{p-1}\alpha_j-\alpha_{p-1}
\]
we have $T_{A_{p-1}}\subset R^\#$. Because $[A_{p-1}:T(A_{p-1})]=p$ and $R^\#$ is strictly contained in $R=A_{p-1}$, we must have $T_{A_{p-1}}= R^\#$.

We have ${1\over p}\sum\limits_{i=1}^{p-1}\omega_i\in (T(A_{p-1}))^{\vee}$. We can check that:
\begin{equation}
\label{equation: action on discriminant nontrivial}
   g({1\over p}\sum\limits_{i=1}^{p-1}\omega_i)-{1\over p}\sum\limits_{i=1}^{p-1}\omega_i=-\omega_1\notin A_{p-1}.
\end{equation}
Hence  $H$ acts nontrivially on the discriminant group of $T(A_{p-1})$.
\end{proof}

Suppose $(H,R)$ is a $p$-root pair with $R=A_{p-1}^k$, where $k$ is a positive integer. We write $R=R_1\oplus \cdots\oplus R_k$ with each $R_i$ isomorphic to $A_{p-1}$. Then $(H_i, R_i)$ is a $p$-root pair by Lemma \ref{lemma: component inherit the structure}. By Lemma \ref{lemma: action on discriminant of T(A) is nontrivial}, $H_i\cap W(R_i)$ contains a generator of order $p$, which gives rise to a sublattice $T(R_i)\subset R_i$. Then we have the following lemma:
\begin{lemma}
\label{lemma: action on discriminant of M is nontrivial}
Under the above setup, we have $\oplus_{i=1}^k T(R_i)\subset R^\#$, and for any even lattice $M$ with $\oplus_{i=1}^k T(R_i) \subset M\subsetneq R$ which is preserved by the action of $H$, the induced action of $H$ on $A_M$ is nontrivial.
\end{lemma}
\begin{proof}
By Lemma \ref{lemma: action on discriminant of T(A) is nontrivial}, we have $T(R_i)=R_i^\#$ with respect to the $p$-root pair $(H_i, R_i)$. Therefore, $R^\#$ contains $\oplus_{i=1}^k T(R_i)$ as a sublattice. Now suppose $M$ to be a lattice with $\oplus_{i=1}^k T(R_i)\subset M\subsetneq R=\oplus_{i=1}^k R_i$ and is preserved by the action of $H$. We have
\begin{equation*}
\oplus_{i=1}^k T(R_i) \subset M \subsetneq \oplus_{i=1}^k R_i \subsetneq \oplus_{i=1}^k R_i^\vee \subsetneq M^\vee \subset \oplus_{i=1}^k T(R_i)^\vee.
\end{equation*}
We can take an element $x=(x_1, \cdots, x_k)\in M^\vee$, such that $x_1\notin R_1^\vee$. Take $h$ a generator of $H_1\cap W(R_1) \le H$. From Equation \eqref{equation: action on discriminant nontrivial}, we know $h(x_1)-x_1\notin R_1$. Therefore, $h(x)-x\notin R$. This implies that the action of $h$ on $A_M$ is nontrivial.
\end{proof}

The following theorem is obtained by Dolgachev--Keum \cite[Theorem 2.1]{DK09} through a geometric argument using fixed loci and fibrations. For supersingular cases, we provide a lattice theoretic proof.

\begin{theorem}[Dolgachev--Keum]\label{thm:new DK}
    Let $X$ be a supersingular K3 surface. Then if  $p\ge 13$, for any finite group $G$ acting symplectically and faithfully on $X$, we have $(p,|G|)=1$.
\end{theorem}
\begin{proof}
When $\sigma(X)=1$, this follows from the classfication result in \cite{OS24}. We now deal with the cases that $\sigma(X)\ge 2$. We argue by contradiction. 
     
    Consider the sequence \eqref{eq:decomposition}. By Proposition \ref{proposition: G_L}, we have shown that $(G_L,S_{G_L}(N_{p,1}))$ is a H\"ohn--Mason pair, then by the H\"ohn--Mason list, $G_L$ is tame. Hence $G\ne G_L$, and $R$ is nontrivial.
    
    Without loss of generality, we assume that $G\cong \mathbb{Z}/p\mathbb{Z}$. By Proposition \ref{prop:no D}, there is no component of type D. By Lemma \ref{proposition: p>=7 no type E}, there is no component of type E. Then all the possible $R$ is of type $A$. Again by Lemma \ref{lemma: component inherit the structure} and Proposition \ref{prop: type A}, $R$ has to be irreducible. Hence $R=A_{p-1}$.

    Suppose $\rk(S_G(N_{p,\sigma}))= \rk(A_{p-1})$. Since $A_{A_{p-1}}\cong\mathbb{Z}/p\mathbb{Z}$ and $S_G(N_{p,\sigma})$ is rootless, then $S_G(N_{p,\sigma})=T(A_{p-1})$. By Lemma \ref{lemma: action on discriminant of M is nontrivial}, the action of $G_R$ on $A_{S_G(N_{p,\sigma})}$ is nontrivial, contradiction. Therefore,  $\rk(S_G(N_{p,\sigma}))> \rk(A_{p-1})$. The action of $g$ on $S_G(N_{p,\sigma})$ has no fixed vector, which implies $(p-1) \big{|} \rk(S_G(N_{p,\sigma}))$. Then $\rk(S_G(N_{p,\sigma}))\ge 2(p-1)\ge 24$, contradiction.
\end{proof}
\subsection{Upper Bounds of Wild Degrees}\label{sec:upper bound}
We denote the exponent of $p$ in $|G|$ by $\nu_p(G)$, and call it the wild degree of $G$. In this subsection we aim to give an upper bound of the wild degree for a finite group $G$ of symplectic automophisms of a supersingular K3 surface in characteristic $p=11, 7, 5, 3$. We start from the following lemma which specializes results on $p$-root pairs to $R\subset S_G(N_{p,1})$.
\begin{lemma}\label{lem:specialization}
    \begin{enumerate}
        \item For $p=11$, the only possible irreducible component of $R$ is $A_{10}$.
        \item For $p=7$, the only possible irreducible component of $R$ is $A_{6}$.
        \item For $p=5$, all possible irreducible components of $R$ are: $A_{4}, E_8$.
        \item For $p=3$, all possible irreducible components of $R$ are: $A_2, A_8, D_4, E_6,E_8$.
    \end{enumerate}
\end{lemma}
\begin{proof}
    For $p=11,7$, by Proposition \ref{prop:no D}, \ref{prop: p>=5 no d4} and \ref{proposition: p>=7 no type E}, all components of $R$ are of type A. Since $\rank R\le 22$, by Lemma \ref{lemma: component inherit the structure} and Proposition \ref{prop: type A}, each component of $R$ is isomorphic to $A_{p-1}$. 

    For $p=5$, by Proposition \ref{prop:no D}, \ref{prop: p>=5 no d4}, there is no irreducible component of type $D$. Then the conclusion follows from Lemma \ref{lemma: component inherit the structure}, Proposition \ref{prop: type A}, \ref{prop: E contributes 1}.

  For $p=3$, by Proposition \ref{prop:no D}, Lemma \ref{lem:p=3,D4}, $D_4$ is the only possibility of type D. The remaining conclusion follows from Lemma \ref{lemma: E}, \ref{lemma: component inherit the structure} and Proposition \ref{prop: type A}.   
\end{proof}
\begin{proposition}\label{proposition: precise estimate of 11-order}
For $p=11$, we have $\nu_p(G)\le 1$.
\end{proposition}
\begin{proof}
   By Lemma \ref{lem:specialization}, $R\cong A_{10}$ or $R\cong A_{10}\oplus A_{10}$. By Lemma \ref{lem: pseudo A}, $\nu_p(G_R)\le 2$. 

Suppose $\nu_p(G_R)=2$, then $R=R_1 \oplus R_2$ with $R_1\cong R_2\cong A_{10}$ and $G_R\cap W(R)=(\BZ/11\BZ)^2$. Take generators $g_1, g_2$ for $G_R\cap W(R)$, such that $g_1$ acts trivially on $R_2$, and $g_2$ acts trivially on $R_1$. Let $\alpha_i=v_i-v_{i+1}$ ($1\le i\le 10$) be simple roots for $R_1$, such that $g_1(v_i)=v_{i+1}$ for $1\le i\le 10$ and $g_1(v_{11})=v_1$. Then we can define $T(R_1)$. Similarly we can define $T(R_2)$. Then $R^\#=T(R_1)\oplus T(R_2)$. 

 The action of $g$ on $S_G(N_{p,\sigma})$ has no fixed vector, which implies $(p-1)\big{|} \rk(S_G(N_{p,\sigma})$. We must have $\rank(S_G(N_{p,1}))=20$. Since $q_R=11^{+2}$, $R$ cannot have nontrivial saturation. Thus $S_G(N_{p,1})=R$. Hence we have:
 \[
 R^{\#}\subset S_G(N_{p,\sigma})\subsetneq R.
 \]
 However by Lemma \ref{lemma: action on discriminant of M is nontrivial}, the action of $G_R$ on $A_{S_G(N_{p,\sigma})}$ is nontrivial, contradiction.
Therefore, $\nu_p(G_R)\le 1$.
     
By Proposition \ref{prop:emb leech}, the Leech pair $(G_L, S_{G_L}(N_{p,1}))$ is a subpair of $(\Co_0, \BL)$. Since $G_L$ acts trivially on $R$, by the rank $\ge 10$ cases of H\"ohn-Mason \cite[Table 1]{HM16},  we know that $|G_L|$ is coprime to $p$. Hence we can conclude that $\nu_p(G)\le 1$.
\end{proof}
\begin{proposition}
\label{proposition: estimate of p-order}
    We have the following upper bounds of the wild degree of $G$:
    \begin{enumerate}
        \item For $p=7$, $\nu_p(G)\le 3$.
        \item For $p=5$, $\nu_p(G)\le 6$.
        \item For $p=3$, $\nu_p(G)\le 14$.
    \end{enumerate}
    Upper bounds listed here do not depend on the Artin invariant $\sigma$.
\end{proposition}
\begin{proof}
Suppose $p=7$. By Lemma \ref{lem:specialization},
    \[
   R\cong A_6, \;\text{or} \; R\cong A_6\oplus A_6, \;\text{or}\; R\cong A_6\oplus A_6\oplus A_6
    \]
    
    If $R\cong A_6$, then $\nu_p(G_R)\le 1$. By Proposition \ref{proposition: G_L}, the Leech pair $(G_L, S_{G_L}(N_{p,1}))$ is a subpair of $(\Co_0, \BL)$. Since $G_L$ acts trivially on $R$,  by the rank $\le 18$ cases of Table \ref{table: sublist of HM}, we know that the $\nu_p(G_L)\le 1$ (the only case with $\nu_p(G_L)=1$ is case 52). Hence we can conclude that $\nu_p(G)\le 2$ in this case.
    
    If $R\cong A_6\oplus A_6$, then $\nu_p(G_R)\le 2$. By Proposition \ref{proposition: G_L}, the Leech pair $(G_L, S_{G_L}(N_{p,1}))$ is a subpair of $(\Co_0, \BL)$. Since $G_L$ acts trivially on $R$,  by the rank $\le 12$ cases of Table \ref{table: sublist of HM},  we know that $|G_L|$ is coprime to $p$. Hence we can conclude that $\nu_p(G)\le 2$ in this case.

    If $R\cong A_6\oplus A_6\oplus A_6$, $\nu_p(G_R)\le 3$. Again, $G_L$ acts trivially on $R$,  by the rank $\le 6$ cases of Table \ref{table: sublist of HM},  we know that $|G_L|$ is coprime to $p$. Hence $\nu_p(G)\le 3$ in this case.

    Suppose $p=5$. By Lemma \ref{lem:specialization}, there are two possibilities, components of $R$ can be of type $A_4$ or $E_8$. By Proposition \ref{prop: E contributes 1}, each $E_8$ component contributes at most one copy of $\BZ/5\BZ$. By Proposition \ref{prop:emb leech}, the Leech pair $(G_L, S_{G_L}(N_{p,1}))$ is a subpair of $(\Co_0, \BL)$. By the $\rk \le 16$ cases in Table \ref{table: sublist of HM}, $|G_L|$ is coprime to $p$. Hence we can conclude that if there exist $E_8$ components, $\nu_p(G)\le 4$ and the equality holds only if $R=E_8\oplus A_4\oplus A_4\oplus A_4$.

    Now we assume that there is no $E_8$ component. Then as in $p=7$, we can analyze $\nu_p(G_R), \mu_p(G_L)$ by the number of $A_4$ components in $R$ and Table \ref{table: sublist of HM}. To reach the maximal $\nu_p$, $R\cong A_4\oplus A_4\oplus A_4\oplus A_4\oplus A_4$. Considering the possible extra $\BZ/5\BZ$ given by the permutation, we have $\nu_p(G)\le 6$.

Finally suppose $p=3$. By Lemma \ref{lem:specialization}, all possible irreducible components of $R$ are
$A_2,A_8, D_4,E_6,E_8$. We also have
\begin{enumerate}
    \item If $(G, A_2)$ is a $p$-root pair, $\nu_p(G)\le 1$.
    \item If $(G, A_8)$ is a $p$-root pair, $\nu_p(G)\le 4$.
    \item If $(G, D_4)$ is a $p$-root pair, $\nu_p(G)\le 1$ (see Lemma \ref{lem:p=3,D4}).
    \item If $(G, E_6)$ is a $p$-root pair, $\nu_p(G)\le 4$.
    \item If $(G, E_8)$ is a $p$-root pair, $\nu_p(G) \le 5$.
\end{enumerate}
Compare $(G, A_2\oplus A_2\oplus A_2\oplus A_2)$ with $(G, A_8)$, there can be an extra $S_4$ acting on $A_2\oplus A_2\oplus A_2\oplus A_2$. Hence to reach the maximal possible wild degree, we only need to consider $A_2, E_6,E_8$ components. 
\begin{enumerate}
    \item If $\rk R=2$, $R=A_2$, then $\nu_p(G_R)\le 1$. By rank $\le 22$ cases in Table \ref{table: sublist of HM}, $\nu_p(G_L)\le 6$ (see case 101). Hence $\nu_p(G)\le 7$.
    \item If $\rk R=4$, we can choose $R=A_2\oplus A_2$, then $\nu_p(G_R)\le 2$. By rank $\le 20$ cases in Table \ref{table: sublist of HM}, $\nu_p(G_L)\le 6$ (see case 101). Hence $\nu_p(G)\le 8$.
    \item If $\rk R=6$, we have $\nu_p(G_R)\le 4$. By rank $\le 18$ cases in Table \ref{table: sublist of HM}, $\nu_p(G_L)\le 5$ (see case 35). Hence $\nu_p(G)\le 9$.
    \item If $\rk R=8$, we have $\nu_p(G_R)\le 5$. By rank $\le 16$ cases in Table \ref{table: sublist of HM}, $\nu_p(G_L)\le 1$. Hence $\nu_p(G)\le 6$.
    \item If $\rk R=10$, we have $\nu_p(G_R)\le 6$ ($R=A_2^{\oplus 5}$). By rank $\le 14$ cases in Table \ref{table: sublist of HM}, $\nu_p(G_L)\le 1$. Hence $\nu_p(G)\le 7$.
    \item If $\rk R=12$, we have $\nu_p(G_R)\le 8$. By rank $\le 12$ cases in Table \ref{table: sublist of HM}, $\nu_p(G_L)\le 1$. Hence $\nu_p(G)\le 9$.
    \item If $\rk R\ge 14$, $\nu_p(G_R)\le 14$, the equality holds when $R=A_2^{\oplus 10}$.   By rank $\le 10$ cases in Table \ref{table: sublist of HM}, $\nu_p(G_L)=0$. Hence $\nu_p(G)\le 14$.
\end{enumerate}
To conclude, we have $\nu_p(G)\le 14$.
\end{proof}

\appendix
\section{H\"ohn--Mason Pairs Admitting Primitive Embedding into $N_{p,1}$}
\label{appendix}
For a Leech pair $(G,S)$ from H\"ohn--Mason list, we need to justify for which odd prime $p$ is $S$ a primitive sublattice of $N_{p,1}$.

Recall that $q_{N_{p,1}}=p^{+2}$ if $p\equiv 3 \; (\md\; 4)$ and $q_{N_{p,1}}=p^{-2}$ if $p\equiv 1 \;(\md\; 4)$. We already proved that for $(G,S)$ there exists a primitive embedding of $S$ into $N_{p,1}$ if and only if there exists an embedding of $S\oplus D_{p,1}$ into $\II_{9,41}$ such that the image of $S$ is primitive. The algorithm is described in the proof of Proposition \ref{proposition: superspecial list}. 

\begin{longtable}
{|p{.04\textwidth} |p{.05\textwidth} | p{.27\textwidth} | p{.10\textwidth} | p{.18\textwidth} | p{.25\textwidth} |}
\hline
No. & rk(S) & Group & Order & $q_S$ & p \\\hline
$1$ & $0$ & $1$ & $1$ & $1$ & any \\
$2$ & $8$ & $2$ & $2$ & $2_{\II}^{+8}$ & any \\
$3$ & $12$ & $2^2$ & $4$ & $2_{\II}^{-6} 4_{\II}^{-2}$ & any \\
$4$ & $12$ & $3$ & $3$ & $3^{+6}$ & any \\
$6$ & $14$ & $2^3$ & $8$ & $2_{\II}^{+6} 4_2^{+2}$  & any \\
$7$ & $14$ & $S_3$ & $6$ & $2_{\II}^{-2} 3^{+5}$ & any \\
$9$ & $14$ & $4$ & $4$ & $2_2^{+2} 4_{\II}^{+4}$ & any \\
$10$ & $15$ & $2^4$ & $16$ & $2_{\II}^{+6}8_1^{+1}$  & any \\
$13$ & $15$ & $D_8$ & $8$ & $4_1^{+5}$  & any \\
$15$ & $16$ & $A_{3,3}$ & $18$ & $3^{+4}9^{-1}$ & any \\
$17$ & $16$ & $\Gamma_2 a_1=2\times D_8$ & $16$ & $2_{\II}^{+2} 4_0^{+4}$  & any \\
$18$ & $16$ & $D_{12}$ & $12$ & $2_{\II}^{+4}3^{+4}$ & any \\
$19$ & $16$ & $A_4$ & $12$ & $2_{\II}^{-2}4_{\II}^{-2}3^{+2}$ & any \\
$20$ & $16$ & $D_{10}$ & $10$ & $5^{+4}$ & any \\
$24$ & $17$ & $2^4:3$ & $48$ & $2_{\II}^{-4}8_1^{+1}3^{-1}$  & any \\
$26$ & $17$ & $\Gamma_4 a_1=2^4.2$ & $32$ & $2_{\II}^{+2}4_6^{+2}8_1^{+1}$  & any \\
$28$ & $17$ & $\Gamma_5 a_1=Q_8 * Q_8$ & $32$ & $4_7^{+5}$  & any \\
$29$ & $17$ & $S_4$ & $24$ & $4_3^{+3}3^{+2}$  & any \\
$31$ & $17$ & $Q_8$ & $8$ & $2_7^{-3}8_{\II}^{-2}$  & any \\
$35$ & $18$ & $3^{1+4}:2$ & $486$ & $3^{+5}$ & $3$ \\
$37$ & $18$ & $4^2.A_4$ & $192$ & $2_{\II}^{-2}8_6^{-2}$ 
 & any \\
$38$ & $18$ & $2^4.D_6$ & $96$ & $2_{\II}^{-2}4_7^{+1}8_1^{+1}3^{-1}$ 
 & any \\
$39$ & $18$ & $A_{4,3}$ & $72$ & $4_{\II}^{-2}3^{-3}$ 
 & any \\
$43$ & $18$ & $\Gamma_{25}a_1$ & $64$ & $4_5^{+3}8_1^{+1}$ 
 & any \\
$44$ & $18$ & $A_5$ & $60$ & $2_{\II}^{-2}3^{+1}5^{-2}$ 
 & any \\
$45$ & $18$ & $2\times S_4$ & $48$ & $2_{\II}^{+2}4_2^{+2}3^{+2}$ 
 & any \\
$46$ & $18$ & $3^2.4$ & $36$ & $2_6^{-2}3^{+2}9^{-1}$ 
 & any \\
$47$ & $18$ & $S_{3,3}$ & $36$ & $2_{\II}^{-2}3^{+3}9^{-1}$ 
 & any \\
$52$ & $18$ & $F_{21}$ & $21$ & $7^{+3}$ & any \\
$53$ & $18$ & $\mathrm{Hol}_4$ & $20$ & $2_6^{-2}5^{+3}$ & any \\
$55$ & $18$ & $\Gamma_3 a_2=\mathrm{SD}_{16}$ & $16$ & $2_7^{+1}4_7^{+1}8_{\II}^{+2}$ & any \\
$68$ & $19$ & $3^{1+4}:2.2$ & $972$ & $2_1^{+1}3^{-4}$ & 3 \\
$69$ & $19$ & $M_{20}$ & $960$ & $2_{\II}^{-2}8_1^{+1}5^{-1}$ & any \\
$70$ & $19$ & $4^2.S_4$ & $384$ & $4_7^{+1}8_6^{+2}$ & any \\
$72$ & $19$ & $A_6$ & $360$ & $4_5^{-1}3^{+2}5^{+1}$  & any \\
$73$ & $19$ & $A_{4,4}$ & $288$ & $2_{\II}^{+2}8_1^{+1}3^{+2}$ & any \\
$74$ & $19$ & $2^4:D_{12}$ & $192$ & $4_2^{-2}8_1^{+1}3^{-1}$ & any \\
$76$ & $19$ & $(Q_8 * Q_8):S_3$ & $192$ & $4_7^{-3}3^{+1}$ & any \\
$77$ & $19$ & $L_2(7)$ & $168$ & $4_1^{+1}7^{+2}$ & any \\
$82$ & $19$ & $S_5$ & $120$ & $4_3^{-1}3^{+1}5^{-2}$ & any \\
$84$ & $19$ & $M_9$ & $72$ & $2_7^{-3}3^{-1}9^{-1}$ & any \\
$85$ & $19$ & $3^2.D_8$ & $72$ & $4_1^{+1}3^{+2}9^{-1}$ & any \\
$87$ & $19$ & $T_{48}$ & $48$ & $2_7^{+1}8_{\II}^{-2}3^{-1}$ & any \\
$101$ & $20$ & $3^4:A_6$ & $29160$ & $3^{+2}9^{+1}$ & $3$ \\
$102$ & $20$ & $L_3(4)=M_{21}$ & $20160$ & $2_{\II}^{-2}3^{-1}7^{-1}$ & $3,7$ or $({21\over p})=-1$ \\
$106$ & $20$ & $2^4:A_6$ & $5760$ & $4_5^{-1}8_1^{+1}3^{+1}$ & $3$ or $({6\over p})=-1$ \\
$108$ & $20$ & $A_7$ & $2520$ & $3^{+1}5^{+1}7^{+1}$ & $3,5,7$ or $({105\over p})=-1$ \\
$109$ & $20$ & $3^{1+4}:2.2^2$ & $1944$ & $2_2^{+2}3^{+3}$ & $3$ \\
$110$ & $20$ & $2^4:S_5=M_{20}:2$ & $1920$ & $4_3^{-1}8_1^{+1}5^{-1}$ & $5$ or $({10\over p})=-1$ \\
$111$ & $20$ & $2^3:L_2(7)$ & $1344$ & $4_2^{+2}7^{+1}$ & $7$ or $({7\over p})=-1$ \\
$112$ & $20$ & $Q(3^2:2)=2^2.A_{4,4}$ & $1152$ & $8_6^{-2}3^{-1}$ & $3$ or $({3\over p})=-1$ \\
$118$ & $20$ & $S_6$ & $720$ & $2_{\II}^{-2}3^{+2}5^{+1}$ & $3,5$ or $({5\over p})=-1$ \\
$119$ & $20$ & $M_{10}$ & $720$ & $2_3^{-1}4_7^{+1}3^{-1}5^{+1}$ & $3,5$ or $({30\over p})=-1$ \\
$120$ & $20$ & $L_2(11)$ & $660$ & $11^{+2}$ & $11$ \\
$121$ & $20$ & $2^4:(S_3\times S_3)$ & $576$ & $4_7^{+1}8_1^{+1}3^{+2}$ & $({2\over p})=-1$ \\
$122$ & $20$ & $5^{1+2}:4$ & $500$ & $5^{+3}$ & 5 \\
$128$ & $20$ & $(3\times A_5):2$ & $360$ & $3^{-2}5^{-2}$ & $5$ \\
$129$ & $20$ & $2\times L_2(7)$ & $336$ & $2_{\II}^{+2}7^{+2}$ & $7$ \\
$134$ & $20$ & $3^2:\mathrm{SD}_{16}$ & $144$ & $2_1^{+1}4_1^{+1}3^{-1}9^{-1}$ & $({6\over p})=-1$ \\
$163$ & $21$ & $U_4(3)$ & $3265920$ & $4_7^{+1}3^{+2}$  & $3$ \\
$165$ & $21$ & $M_{22}$ & $443520$ & $4_5^{-1}11^{+1}$ & $11$ \\
$167$ & $21$ & $U_3(5)$ & $126000$ & $2_7^{+1} 5^{-2}$ & $5$ \\
$170$ & $21$ & $L_3(4).2$ & $40320$ & $4_3^{-1}3^{-1}7^{-1}$ & $7$ \\
$172$ & $21$ & $2^4.A_7=2^4:A_7$ & $40320$ & $8_1^{+1}7^{+1}$  & $7$ \\
$175$ & $21$ & $A_8$ & $20160$ & $4_1^{+1}3^{+1}5^{+1}$ & $5$ \\
$182$ & $21$ & $M_{11}$ & $7920$ & $2_7^{+1}3^{-1}11^{-1}$ & $11$ \\
$183$ & $21$ & $[2^4 3].S_5=2^4:(3\times S_5)$ & $5760$ & $8_1^{+1}3^{-1}5^{-1}$ & $5$ \\
\hline\hline
\caption{H\"ohn--Mason Pairs Admitting Primitive Embedding into $N_{p,1}$}
\label{table: sublist of HM}
\end{longtable}

\begin{remark}
There are some typos about $2$-parts in discriminants in H\"ohn--Mason. We corrects the typos we found by comparing \cite{HM16} to \cite{hashimoto2012finite}. For case 31, $2_3^{+3}8_{\II}^{-2}$ is corrected to $2_7^{-3}8_{\II}^{-2}$. For case 46, $2_2^{+2}3^{+2}9^{-1}$ is corrected to $2_6^{-2}3^{+2}9^{-1}$. For case 53, $2_2^{+2}5^{+3}$ is corrected to $2_6^{-2}5^{+3}$. For case 55, $2_5^{+1}4_1^{+1}8_{\II}^{+2}$ is corrected to $2_7^{+1}4_7^{+1}8_{\II}^{+2}$. For case 70, $4_3^{+1}8_2^{+2}$ is corrected to $4_7^{+1}8_6^{+2}$. For case 84, $2_2^{+3}3^{-1}9^{-1}$ is corrected to $2_7^{-3}3^{-1}9^{-1}$. For case 119, $2_5^{+1}4_1^{+1}3^{-1}5^{+1}$ is corrected to $2_3^{-1}4_7^{+1}3^{-1}5^{+1}$.

In the case $(G, q_S)\cong ([2^4 3].S_5, 8_1^{+1} 3^{-1} 5^{-1})$, a more precise description of $G$ has been discussed in \cite[Remark 7.2]{OS24}.
\end{remark}

\bibliographystyle{alpha}
\bibliography{ref}

\begin{thebibliography}{GHV12}

\bibitem[AM77]{AM77}
M.F. Artin and B.~Mazur.
\newblock Formal groups arising from algebraic varieties.
\newblock {\em Annales scientifiques de l'\'Ecole Normale Sup\'erieure}, Ser.
  4, 10(1):87--131, 1977.

\bibitem[Art74]{Art74}
M.~Artin.
\newblock Supersingular $ k3 $ surfaces.
\newblock In {\em Annales scientifiques de l'{\'E}cole Normale Sup{\'e}rieure},
  volume~7, pages 543--567, 1974.

\bibitem[BL18]{BL18}
D.~Bragg and M.~Lieblich.
\newblock {Twistor spaces for supersingular K3 surfaces}.
\newblock {\em arXiv preprint arXiv:1804.07282}, 2018.

\bibitem[CS99]{conway1999spherepackings}
J.~H. Conway and N.~J.~A. Sloane.
\newblock {\em Sphere packings, lattices and groups}, volume 290 of {\em
  Grundlehren der Mathematischen Wissenschaften}.
\newblock Springer-Verlag, New York, third edition, 1999.
\newblock With additional contributions by E. Bannai, R. E. Borcherds, J.
  Leech, S. P. Norton, A. M. Odlyzko, R. A. Parker, L. Queen and B. B. Venkov.

\bibitem[DK09a]{DK09}
I.~Dolgachev and J.~Keum.
\newblock Finite groups of symplectic automorphisms of k3 surfaces in positive
  characteristic.
\newblock {\em Annals of mathematics}, pages 269--313, 2009.

\bibitem[DK09b]{DK09J}
Igor~V. Dolgachev and JongHae Keum.
\newblock {$K3$} surfaces with a symplectic automorphism of order 11.
\newblock {\em J. Eur. Math. Soc. (JEMS)}, 11(4):799--818, 2009.

\bibitem[GHV12]{gaberdiel2012symmetries}
M.R. Gaberdiel, S.~Hohenegger, and R.~Volpato.
\newblock {Symmetries of K3 sigma models}.
\newblock {\em Communications in Number Theory and Physics}, 6(1):1--50, 2012.

\bibitem[GL11]{GL11}
R.L. Griess and C.H. Lam.
\newblock A moonshine path for {$5A$} and associated lattices of ranks 8 and
  16.
\newblock {\em J. Algebra}, 331:338--361, 2011.

\bibitem[Has12]{hashimoto2012finite}
K.~Hashimoto.
\newblock Finite symplectic actions on the k3 lattice.
\newblock {\em Nagoya Mathematical Journal}, 206:99--153, 2012.

\bibitem[HM16]{HM16}
G.~H{\"o}hn and G.~Mason.
\newblock The 290 fixed-point sublattices of the leech lattice.
\newblock {\em Journal of algebra}, 448:618--637, 2016.

\bibitem[Huy16]{Huy16}
D.~Huybrechts.
\newblock {\em Lectures on K3 surfaces}, volume 158.
\newblock Cambridge University Press, 2016.

\bibitem[Jan16]{Jan16}
J.~Jang.
\newblock The representations of the automorphism groups and the frobenius
  invariants of k3 surfaces.
\newblock {\em The Michigan Mathematical Journal}, 65(1):147--163, 2016.

\bibitem[Jan17]{jang2017lifting}
J.~Jang.
\newblock A lifting of an automorphism of a k3 surface over odd characteristic.
\newblock {\em International Mathematics Research Notices}, 2017(6):1787--1804,
  2017.

\bibitem[Kac90]{Kac90}
V.G. Kac.
\newblock {\em Infinite-dimensional {L}ie algebras}.
\newblock Cambridge University Press, Cambridge, third edition, 1990.

\bibitem[Kon98]{Kon98}
S.~Kond{\=o}.
\newblock {Niemeier lattices, Mathieu groups, and finite groups of symplectic
  automorphisms of K3 surfaces}.
\newblock {\em Duke mathematical journal}, 92(3):593--603, 1998.

\bibitem[Lie13]{Lie13}
C.~Liedtke.
\newblock Algebraic surfaces in positive characteristic.
\newblock In {\em Birational geometry, rational curves, and arithmetic}, pages
  229--292. Springer, 2013.

\bibitem[Lie16]{Lie16}
C.~Liedtke.
\newblock Lectures on supersingular k3 surfaces and the crystalline {Torelli}
  theorem.
\newblock In {\em K3 surfaces and their moduli}, pages 171--235. Springer,
  2016.

\bibitem[LM18]{LM18}
M.~Lieblich and D.~Maulik.
\newblock A note on the cone conjecture for {K}3 surfaces in positive
  characteristic.
\newblock {\em Math. Res. Lett.}, 25(6):1879--1891, 2018.

\bibitem[LZ22]{laza2022symplectic}
R.~Laza and Z.~Zheng.
\newblock Automorphisms and periods of cubic fourfolds.
\newblock {\em Math. Z.}, 300(2):1455--1507, 2022.

\bibitem[MM25]{marquand2025Ogrady10}
L.~Marquand and S.~Muller.
\newblock {Finite groups of symplectic birational transformations of IHS
  manifolds of OG10 type}.
\newblock {\em Forum of Mathematics, Sigma}, 13, 2025.

\bibitem[Muk88]{Muk88}
S.~Mukai.
\newblock Finite groups of automorphisms of k3 surfaces and the {Mathieu}
  group.
\newblock {\em Inventiones mathematicae}, 94(1):183--221, 1988.

\bibitem[Nik79a]{Nik79}
V.~V. Nikulin.
\newblock Integral symmetric bilinear forms and some of their applications.
\newblock {\em Izvestiya Rossiiskoi Akademii Nauk. Seriya Matematicheskaya},
  43(1):111--177, 1979.

\bibitem[Nik79b]{nikulin1979finite}
V.V. Nikulin.
\newblock {Finite groups of automorphisms of K\"ahlerian K3 surfaces}.
\newblock {\em Trudy Moskovskogo Matematicheskogo Obshchestva}, 38:75--137,
  1979.

\bibitem[NT98]{NT98}
I.~Naruki and N.~Tsuchida.
\newblock On the signature for finite quadratic forms.
\newblock {\em Mem. Inst. Sci. Engrg. Ritsumeikan Univ.}, (57):33--48 (1999),
  1998.

\bibitem[Nyg81]{Nyg80}
Niels~O. Nygaard.
\newblock Higher de\thinspace {R}ham-{W}itt complexes of supersingular {$K3$}\
  surfaces.
\newblock {\em Compositio Math.}, 42(2):245--271, 1980/81.

\bibitem[Ogu79]{Ogu79}
A.~Ogus.
\newblock Supersingular k3 crystals.
\newblock {\em Ast{\'e}risque}, 64:3--86, 1979.

\bibitem[Ogu83]{Ogu83}
A.~Ogus.
\newblock A crystalline {Torelli} theorem for supersingular k3 surfaces.
\newblock In {\em Arithmetic and geometry}, pages 361--394. Springer, 1983.

\bibitem[OS24]{OS24}
H.~Ohashi and M.~Sch{\"u}tt.
\newblock Finite symplectic automorphism groups of supersingular k3 surfaces.
\newblock {\em arXiv preprint arXiv:2405.06341}, 2024.

\bibitem[RS81]{RS81}
A.N. Rudakov and I.R. Shafarevich.
\newblock Surfaces of type k3 over fields of finite characteristic.
\newblock {\em Itogi Nauki i Tekhniki. Seriy ``Sovremennye Problemy Matematiki.
  Noveishie Dostizheniya"}, 18:115--207, 1981.

\bibitem[Zhe25]{zheng2025leechlemma}
Z.~Zheng.
\newblock {A Lemma on Leech-like Lattices}.
\newblock {\em arXiv preprint arXiv:2507.10414}, 2025.

\end{thebibliography}
\Addresses
\end{document}